\documentclass[a4paper]{article}

\usepackage{amssymb}
\usepackage{amsfonts}
\usepackage{amssymb,amsmath,latexsym, amsthm,tensor}
\usepackage{mathrsfs}
\usepackage{enumerate}
\usepackage{savesym}

\usepackage[english]{babel}
\usepackage[utf8x]{inputenc}
\usepackage{amsmath}
\usepackage{graphicx}
\usepackage[colorinlistoftodos]{todonotes}

\usepackage[labelfont=bf]{caption}

\usepackage[english]{babel}
\usepackage[utf8x]{inputenc}
\usepackage[colorinlistoftodos]{todonotes}
\usepackage{caption}
\usepackage{subcaption}
\savesymbol{AND}
\usepackage{algorithm}
\usepackage[noend]{algpseudocode}

\newcommand{\R}{\mathbb{R}}

\newtheorem*{remark}{Remark}
\newtheorem{theorem}{Theorem}

\newtheorem{lemma}[theorem]{Lemma}

\usepackage[symbol]{footmisc}

\title{Generalized Affine Scaling Algorithms for Linear Programming Problems}
\author{Md Sarowar Morshed and Md Noor-E-Alam \\ {\small Department of Mechanical \& Industrial Engineering} \\ {\small Northeastern University } \\
{\small 360 Huntington Avenue, Boston, MA 02115, USA} \\
{\small Email : mnalam@neu.edu}}
\date{}

\begin{document}
\maketitle

\section*{Abstract}

Interior Point Methods are widely used to solve Linear Programming problems.  In this work, we present two primal Affine Scaling algorithms to achieve faster convergence in solving Linear Programming problems. In the first algorithm, we integrate Nesterov's restarting strategy in the primal Affine Scaling method with an extra parameter, which in turn generalizes the original primal Affine Scaling method. We provide the proof of convergence for the proposed generalized algorithm considering long step size. We also provide the proof of convergence for the primal and dual sequence without the degeneracy assumption. This convergence result generalizes the original convergence result for the Affine Scaling methods and it gives us hints about the existence of a new family of methods. Then, we introduce a second algorithm to accelerate the convergence rate of the generalized algorithm by integrating a non-linear series transformation technique. Our numerical results show that the proposed algorithms outperform the original primal Affine Scaling method.

% Enter your abstract

%\KEYWORDS{deterministic inventory theory; infinite linear programming duality;
%  existence of optimal policies; semi-Markov decision process; cyclic schedule}

\textit{\textbf{Key words:}} Linear Programming, Affine Scaling, Nesterov Acceleration, Dikin Process, Shanks Series Transformation

\maketitle
%%%%%%%%%%%%%%%%%%%%%%%%%%%%%%%%%%%%%%%%%%%%%%%%%%%%%%%%%%%%%%%%%%%%%%

% Samples of sectioning (and labeling) in OPRE
% NOTE: (1) \section and \subsection do NOT end with a period
%       (2) \subsubsection and lower need end punctuation
%       (3) capitalization is as shown (title style).
%
%\section{Introduction.}\label{intro} %%1.
%\subsection{Duality and the Classical EOQ Problem.}\label{class-EOQ} %% 1.1.
%\subsection{Outline.}\label{outline1} %% 1.2.
%\subsubsection{Cyclic Schedules for the General Deterministic SMDP.}
%  \label{cyclic-schedules} %% 1.2.1
%\section{Problem Description.}\label{problemdescription} %% 2.

% Text of your paper here

\section{Introduction}
\label{sec:into}

The \textit{Affine Scaling} (AFS) algorithm was introduced by Dikin \cite{dikin:1967}, which remained unnoticed to the \textit{Operations Research} (OR) community until the seminal work of Karmarkar \cite{karmarkar:1984}. Karmarkar's work transformed the research in \textit{Interior Point Methods} (IPMs) and induced a significant development in the theory of IPMs. As a result, several variants of AFS have been studied over the years by researchers (see \cite{jansen:1996}, \cite{barnes:1986}). We refer to the books of Wright \cite{wright:1997}, Ye \cite{ye:1997}, Bertsimas \cite{bertsimas:1997} and Vanderbei \cite{vanderbei:1998} for more comprehensive discussion of these methods.

Apart from the simplicity, convergence analysis of the AFS methods for generalized degenerate setup are considered difficult to analyze. Dikin first published a convergence proof with a non-degeneracy assumption in 1974 \cite{vanderbei:1990}. Both Vanderbei \textit{et al.} \cite{vanderbei:1986} and Barnes \cite{barnes:1986} gave simpler proofs in their global convergence analysis but still assumed primal and dual non-degeneracy. First attempt to break out of the non-degeneracy assumption was made by Adler \textit{et al.} \cite{adler:1991}, who investigated the convergence of continuous trajectories of primal and dual AFS. Subsequently, assuming only dual non-degeneracy, Tsuchiya \cite{tsuchiya:19921} showed that under the condition of step size $\alpha < \frac{1}{8}$, the long-step version of AFS converges globally. In another work, Tsuchiya \cite{tsuchiya:1991} showed that the dual non-degeneracy condition is not a necessary condition for the convergence as assumed previously \cite{tsuchiya:19921}. Moreover, Tsuchiya \cite{tsuchiya:1991} introduced the idea of potential function, a slightly different function than the one provided by Karmarkar \cite{karmarkar:1984}, for the analysis of the local behavior of the AFS near the boundary of the feasible region. Finally, using that potential function \cite{tsuchiya:1991}, Dikin \cite{dikin:1991} and Tsuchiya \textit{et al.} \cite{tsuchiya:1992} provided proofs for the global convergence of degenerate \textit{Linear Programming} (LP) problems with $\alpha \footnote{$\alpha$ is the step size} < \frac{1}{2}$  and $\alpha \leqslant \frac{2}{3}$, respectively. Later, Hall \textit{et al.} \cite{hall:1993} showed that the sequence of dual estimates will not always converge for $\alpha > \frac{2}{3}$. 

As a self-contained paper for a global convergence analysis for AFS, Monteiro \textit{et al.} \cite{monteiro:1993} and Saigal \cite{saigal:1996} provided two simple proofs for the long-step AFS algorithms of degenerate LP's. Subsequently, Saigal \cite{saigal:1997} introduced two step predictor corrector based methods to fasten the convergence of AFS. Besides, the chaotic analysis of AFS was first addressed by Castillo \textit{et al.} \cite{barnes:2001}. Bruin \textit{et al.} \cite{ross:2014} provided a proper chaotic explanation of the so called \textit{Dikin Process} by showing the similarity of it with the logistic family in terms of chaotic behavior. In their work, they showed why the AFS algorithms behave differently when the step size $\alpha$ is close to $\frac{2}{3}$, which in general complies with the chaotic behavior of IPMs analyzed by several other researchers. There has been a significant development in applying the AFS techniques to solve various types of optimization problems: Semi Definite Programming \cite{vanderbei:1999}, Nonlinear Smooth Programming \cite{wang:2009}, Linear Convex Programming \cite{cunha:2011}, Support Vector Machine \cite{maria:2011}, Linear Box Constrained Optimization \cite{wang:2014}, Nonlinear Box Constrained Optimization \cite{huang:2017}. Recently, Kannan \textit{et al.} \cite{kannan:2012} applied the idea of AFS algorithm to generate a random walk for solving LP problems approximately.

In a seminal work, Nesterov \cite{nesterov:1983} proposed an acceleration technique for the \textit{Gradient Descent} that exhibits the worst-case convergence rate of $O(\frac{1}{k^2})$ for minimizing smooth convex functions compared to the original convergence rate of $O(\frac{1}{k})$. Since the inception of Nesterov's work, there has been a body of work done on the theoretical development of first-order accelerated methods (for a detailed discussion see \cite{nesterov:2005}, \cite{nesterov:2013} and \cite{nesterov:2014}). Furthermore, an unified summary of all of the methods of Nesterov can be found in \cite{tseng:2010}.  Recently, Su \textit{et al.} \cite{su:2016} carried out a theoretical analysis on the methods of Nesterov and showed that it can be interpreted as a finite difference approximation of a second-order \textit {Ordinary Differential Equation} (ODE).

\subsection{\textbf{Motivation \& Contribution}}
\label{subsec:motiv}

We have seen from the literature that Nesterov's restarting scheme is very successful in achieving faster convergence for \textit{Gradient Descent} algorithms. However, to the best of our knowledge, the potential opportunity of Nesterov's acceleration has not been yet explored to IPMs to solve LP problems. Motivating by the power of acceleration and to fill the research gap, as a first attempt, in this work we apply acceleration techniques in the Affine Scaling method. Affine Scaling method was chosen for the following three reasons:
\begin{itemize}
    \vspace{-0.2 cm} \item \textit{Historical Importance:} Affine Scaling method was the first discovered Interior Point algorithm discovered by Dikin in 1967 \cite{dikin:1967}. In 1984 Karmarkar introduced potential reduction method (Karmarkar method) based on Dikin's original idea \cite{karmarkar:1984}. After that a significant amount of research was done in the field of IPM algorithms. Dikin's original Affine Scaling method worked as a stepping stone for the development of further IPM algorithms for solving LP problems.
    \vspace{-0.2 cm} \item \textit{Algorithmic Simplicity:} Affine Scaling method uses the simplest primal update whereas Path following/Barrier method and Karmarkar method use penalty function driven update formulas \cite{bertsimas:1997} (i.e., both Barrier method and Karmarkar method have different objectives in the EAP problem shown in equation \ref{eq:n} of Section \ref{sec:conc}). Note that in the EAP problem (equation \eqref{eq:2} and \eqref{eq:n}) Affine Scaling method uses the objective function $c^Td$, whereas Karmarkar method uses the objective $ G(x+d,s)$  and Barrier method uses the objective $ B_\mu (x+d)$ ( where $ G(x,s)$ and $ B_\mu (x)$ are the potential function and barrier function respectively, see Section \ref{sec:conc} for details).
    \vspace{-0.2 cm} \item \textit{Generalized Dikin Process:} The chaotic behavior of Affine Scaling method can be explained by the so called \textit{Dikin Process} \cite{ross:2014}, which has a similarity to the logistic map in \textit{Chaos Theory}. From a \textit{Chaos Theory} perspective, our proposed algorithms may reveal a more generalized versions of \textit{Dikin Process}, which can be represented as dynamic systems (\cite{adler:1991}, \cite{barnes:2001}, \cite{ross:2014}).
\end{itemize}
Based on the above discussion in this work, we propose two algorithms: 1) Generalized AFS and 2) Accelerated AFS (see Section \ref{sec:afs}). In the Generalized AFS algorithm, we propose to use the acceleration scheme discovered by Nesterov \cite{nesterov:1983,nesterov:2013,nesterov:2014,nesterov:2005,nesterov:2012} in the AFS framework to achieve super linear convergence rate as compared to the linear rate shown in \cite{Tsuchiya1996}. Note that, Nesterov's accelerated scheme can be incorporated with Affine Scaling scheme by defining two new sequences $\{u_k\}$ and $\{v_k\} \ \footnote{Representation of this update formula is different than Algorithm 1 but they implies the same steps}$ as following:
\begin{align}
\label{eq:nes}
 & u_k  = \alpha_k v_k + (1-\alpha_k) x_k \quad U_k = \textbf{diag} \left[(u_k)_1, (u_k)_2, ... (u_k)_n\right] \nonumber \\
 & \bar{y_k} = \left(AU_k^2A^T\right)^{-1}AU_k^2c, \ \ \bar{s_k} = c- A^T\bar{y_k} \nonumber\\
 & x_{k+1}  = u_k - \theta_k  \frac{U_k^2\bar{s_k}}{\|U_k\bar{s_k}\|}\\
 & v_{k+1} = \beta_k v_k + (1- \beta_k) u_k - \gamma_k \frac{U_k^2\bar{s_k}}{\|U_k\bar{s_k}\|} \nonumber
\end{align}
In equation (\ref{eq:nes}), instead of using $\nabla f$ as in standard \textit{Gradient Descent} we use $\frac{U_k^2\bar{s_k}}{\|U_k\bar{s_k}\|}$ and $\theta_k$ is the step-size. The main contribution for the above scheme is that it uses appropriate values for the parameters $\alpha_k, \beta_k$ and $\gamma_k$ \footnote{Used as a standard notation for Nesterov's method, throughout the paper we used the same notations with new definition}, which in turn yield better convergence in the context of standard Affine Scaling methods. Surprisingly, the generalized AFS algorithm also follows the same types of convergence scheme as the original AFS method (i.e., feasible step size for convergence is $\alpha \leq \frac{2}{3}$ for AFS, generalized AFS shows the same types of convergence $\alpha + \beta \leq \frac{2}{3}$), which leads to a more generalized method. We then provide a generalized proof of convergence of the proposed generalized AFS under sufficient conditions. To gain further acceleration, in the generalized AFS algorithm, we propose to exploit the entry-wise \textit{Shanks Series Transformation} (SST) to the generalized update of the generalized AFS algorithm. We then carry out rigorous convergence analysis to provide guarantee for  acceleration and show the effectiveness of the proposed algorithms through numerical experiments.

This is the first time, Nesterov's momentum method and SST are used to design better algorithms in the context of general IPMs. We believe our proposed acceleration scheme will facilitate the application of such acceleration techniques in Barrier method/Path following method and Karmarkar method. This scheme will also serve as a foundation for developing accelerated techniques for other more efficient but complex methods in the IPM family. In terms of theoretical contribution, our proposed algorithms reveal interesting properties about the convergence characteristics of Affine Scaling method. Based on our analysis, it is evident that the convergence criterion for AFS '$\alpha \leq 2/3$' is a universal bound as the proposed algorithms satisfy a more generalized bound '$\alpha + \beta \leq 2/3$'. Finally, the proposed algorithms suggest availability of a more general family of numerically efficient and theoretically interesting AFS methods.

The paper is organized as follows. In Section 2, we provide a preliminary idea of the original AFS algorithm, then we describe the proposed variant AFS algorithms. In Section 3, we show the convergence of primal sequence for the proposed algorithms. In section 4, we exploit the convergence rate of the accelerated AFS algorithm. In section 5, we present convergence of the dual sequence under sufficient conditions for both algorithms. In section 6, we present numerical results to demonstrate the effectiveness of the proposed algorithms. Finally, in Section 7, we conclude the paper with future research directions.

\section{Generalization of AFS}
\label{sec:afs}

Affine Scaling method uses the simple idea of reformulation, instead of minimizing over the whole interior, it generates a series of ellipsoids inside the interior of feasible region and moves according to \textit{min.} cost direction. Consider the following standard LP and its dual,
\begin{align}
\label{eq:1}
\textbf{(P):} \quad \min \ \ & c^Tx          &  \textbf{(D):} \quad  \max \ \ & y^Tb             \nonumber\\
 & Ax= b         &  & A^Ty +s = c  \\
& x \geqslant 0   &  & s \geqslant 0     \nonumber    
\end{align}

Let $P = \left\{x \ | \ Ax = b, x \geqslant 0\right\}$ be the primal feasible set, then we call the set $\left\{x \in P \ | \ x > 0\right\}$ as the interior of $P$ and its elements as interior points. The basic idea of Affine Scaling method is that instead of minimizing over $P$, we solve a series of optimization problems over ellipsoids. Starting from an initial strictly feasible solution $x_0 >0$, we form an ellipsoid $S_0$ centered at $x_0$, which is contained in the interior of $P$. Then by minimizing $c^Tx$ over all $x \in S_0$, we find a new interior point $x_1$ and proceed in a similar way until stopping restrictions are satisfied. The AFS method can be easily formulated as the following problem: given a strictly feasible solution $x \in \R^n$, we need to find direction vector $d$ such that $\bar{x} = x+ \alpha d$ for some $\alpha \in (0,1)$ \footnote{$\alpha$ is the step size} and it holds $\bar{x} \in P, \ c^T\bar{x} \leqslant c^Tx$. 

\noindent To integrate acceleration and generalization in the AFS, we have proposed the following two algorithms which are variant of original AFS algorithm:
\begin{enumerate}
\item \textbf{Generalized AFS algorithm (\textbf{GAFS})} 
\item \textbf{Accelerated AFS algorithm (\textbf{AAFS})} 
\end{enumerate}

In \textbf{GAFS}, we propose to use Nesterov's restarting strategy with the original AFS to generalize AFS. To facilitate the convergence process of \textbf{GAFS}, we propose to use entry-wise \textit{Shanks series transformation} (SST) introduced by Shanks \cite{shanks:1965} to \textbf{GAFS}. This integrated algorithm is referred as \textbf{AAFS} in this work. We explain details about these two algorithms below:

\textbf{GAFS:} We followed the Nesterov's restarting strategy and introduced a generalized version of AFS in a way that it will give us the original AFS in the absence of the other parameter. For doing so, we integrate an extra term from the original idea of Nesterov \cite{nesterov:1983}. Here this variant of AFS is refereed as GAFS. In the GAFS, we consider two strict feasible points $x,z$ with $ \ c^Tx < c^Tz $, instead of one point $x$ to find a direction vector $d \in \R^n$ such that $\bar{x} = x+ \alpha d+ \bar{\beta} (x-z)$ for some $\alpha, \beta \in (0,1), \ \bar{\beta} = \frac{\beta}{\|1-X^{-1}z\|_{\infty}}, \ X = \text{diag}(x) $, where $\beta$ is the generalization parameter. It allowed us to reformulate the problem as below:
\begin{align}
\label{eq:2}
\textbf{min} \quad & w = c^Td \nonumber\\
\textbf{s.t} \ & Ad =0 \\
& \| X^{-1}d \| \leqslant \alpha  \nonumber
\end{align}
The above problem is known as \textit{Ellipsoidal Approximating Problem} (EAP), see \cite{saigal:1996} and \cite{tsuchiya:1995} for more detailed information. 

\begin{algorithm*}
\caption{Generalized Affine Scaling Algorithm (\textbf{GAFS})}\label{alg:acc AFS}
\begin{algorithmic}[1]
\State Start with $A,\ b,\ c,\ \epsilon > 0, \ \alpha ,\beta \in (0,1), \ x_0 > 0,  \ k =0$
\Repeat
\begin{align*}
& X_k = \textbf{diag} \left[(x_k)_1, (x_k)_2, ... (x_k)_n\right] \\
& y_k = \left(AX_k^2A^T\right)^{-1}AX_k^2c, \ \ s_k = c- A^Ty_k \\ 
& z_k =
\begin{cases}
\ x_k, & k =0\\
\ x_k + \beta \frac{\delta(x_k)}{\|X_k^{-1}\delta(x_k)\|_{\infty}}, & k > 0
\end{cases}
\end{align*}
\If{$s_k \geq 0$ and $e^T X_k s_k < \epsilon$}
\State \textbf{Stop}; the current $(x_k, y_k)$ are primal and dual $\epsilon$ optimal.
\ElsIf{$ X_k^2s_k \leq 0$}
\State \textbf{Stop}; the problem is unbounded
\Else
\begin{align*}
x_{k+1} = z_k - \alpha \frac{X_k^2s_k}{\|X_ks_k\|}
\end{align*}
\EndIf \Until{\textbf{repeat}}
\end{algorithmic}
\end{algorithm*}

\begin{algorithm*}
\caption{Accelerated Affine Scaling Algorithm (\textbf{AAFS})}\label{alg:shank AFS}
\begin{algorithmic}[1]
\State Start with $A,\ b,\ c,\ \epsilon > 0, \ \alpha \in (0,1), \ x_0 > 0,  \ k =0$
\Repeat
\begin{align*}
& X_k = \textbf{diag} \left[(x_k)_1, (x_k)_2, ... (x_k)_n\right] \\
& y_k = \left(AX_k^2A^T\right)^{-1}AX_k^2c, \ \ s_k = c- A^Ty_k \\
& z_k =
\begin{cases}
\ x_k, & k =0\\
\ x_k + \beta \frac{\delta(x_k)}{\|X_k^{-1}\delta(x_k)\|_{\infty}}, & k > 0
\end{cases} \\
& (B(x_k))_j = \begin{cases}
\ (x_k)_j, & k = 0, 1 \ j = 1,2,...n \\
\  (x_k)_j - \frac{((x_k)_j-(x_{k+1})_j)^2}{(x_k)_j-2(x_{k+1})_j+(x_{k+2})_j} , & k > 1, \  j = 1,2,...n
\end{cases} \\ 
& B_k = \textbf{diag} \left[B(x_k)_1, B(x_k)_2, ... B(x_k)_n\right]
\end{align*}
\If{$s_k \geq 0$ and $e^T B_ks_k < \epsilon$}
\State \textbf{Stop}; the current $(B(x_k), y_k)$ are primal and dual $\epsilon$ optimal.
\ElsIf{$ X_k^2s_k \leq 0$}
\State \textbf{Stop}; the problem is unbounded
\Else
\begin{align*}
x_{k+1} = z_k - \alpha \frac{X_k^2s_k}{\|X_ks_k\|}
\end{align*}
\EndIf \Until{\textbf{repeat}}
\end{algorithmic}
\end{algorithm*}

Since the generalization parameters do not affect the EAP problem, the main properties discussed in \cite{dikin:1991}, \cite{saigal:1996} and \cite{tsuchiya:1995} are valid for \eqref{eq:2} (see \ref{appendix-sec2} for more details). From now on, we denote $\delta(x_k) = x_k-x_{k-1}$. For all $ k \geqslant 0$, we constructed the sequence $z_k$ using the following update with $\alpha, \beta \in (0,1)$ and a strictly feasible point $x_0 > 0$ :
\begin{align}
\label{eq:3a}
z_k =
\begin{cases}
\ x_k, & k =0 \\
\ x_k + \beta \frac{\delta(x_k)}{\|X_k^{-1}\delta(x_k)\|_{\infty}}, & k > 0
\end{cases}
\end{align}
When the stopping criteria is not satisfied ($k \geqslant 0$), $x_{k+1}$ can be calculated using the following formula:
\vspace{-0.4 cm}
\begin{align}
\label{eq:3}
x_{k+1} = z_k - \alpha \frac{X_k^2s_k}{\|X_ks_k\|}
\end{align}
\textbf{AAFS:} In the AAFS, we integrated the SST with the GAFS to gain acceleration. Since the primal sequence generated by the GAFS converges as $ i $ goes to infinity (i.e.,  $\lim_{i \rightarrow \infty} x_i = x^*$, see Section \ref{sec:primal}), it allowed us to  write the following equation:
\begin{align*}
x^*-x_0 = \sum\limits_{i = 0}^{\infty} (x_{i+1}-x_i)
\end{align*}
We denoted the entry-wise partial sum of the right hand side of above equation as $C_{k,j}$ as follows: 
\begin{align*}
C_{k,j} = \sum\limits_{i =0}^{k} (x_{i+1}-x_i)_j
\end{align*}
We see that $C_{k,j}+ (x_0)_j$ converges to $(x^*)_j$ as $k$ goes to infinity for all $j = 1,2,..,n$. This setup allowed us to introduce the entry-wise SST to the sequence $x_k$ generated by the GAFS. In the above algorithm, we define $(B(x_k))_j$ for all $k \geqslant 0$ and $j=1,2...,n$ as follows:
\begin{align}
\label{eq:4}
(B(x_k))_j \overset{\underset{\mathrm{def}}{}}{=}
\begin{cases}
\ (x_k)_j, & k = 0, 1 \\
\ \Big | (x_k)_j - \frac{((x_k)_j-(x_{k+1})_j)^2}{(x_k)_j-2(x_{k+1})_j-(x_{k+2})_j} \Big |, & k > 1
\end{cases}
\end{align}
As $(x_k)_j$ is approximated by $(B(x_k))_j$ for all $ j = 1,2,..,n$, we can modify the stopping criteria of AAFS algorithm with $e^T B_ks_k$.

\section{Convergence of the primal sequence}
\label{sec:primal}
In this section, we provided proof of convergence for the primal sequence $\{x_k\}$ generated by the GAFS and AAFS algorithms discussed in Section 2. We used some Lemmas related to the properties of the sequences $\{x_k\}, \{y_k\}, \{s_k\}, \{d_k \} =\{X_k^2s_k\}$ generated by the GAFS provided in the \ref{appendix-sec2} for proving the next few Theorems of this section. 

We made the following assumptions before providing the proof of convergence of the primal sequence $\{x_k\}$ and the cost function sequence $\{c^Tx_k\}$:
\begin{itemize}
     \vspace{-0.15 cm} \item The \textit{Linear Program} \eqref{eq:1} has at least one interior point feasible solution.
    \vspace{-0.2 cm} \item The objective function $c^Tx$ is not constant over the feasible region of \eqref{eq:1}.
    \vspace{-0.2 cm} \item The matrix $A$ has rank $m$.
    \vspace{-0.2 cm} \item The \textit{Linear Program} has an optimal solution. 
\end{itemize}
\begin{remark}
\label{rem-1}
Note that, here we didn't assume primal and dual non-degeneracy of the LP problem \eqref{eq:1}.
\end{remark}
For step size selection, we considered three well-known function defined for a vector $u$ as
\[\gamma(u) = \max \left\{ u_i \ | \ u_i > 0\right\}, \quad \|u\|_{\infty} = \max_i |u_i|, \quad \|u\|_2 = \sqrt{\sum u_i^2}\]
Whereas the second and third terms are $l_{\infty}$ and $l_2$ norm, respectively. The 1st function is not a norm and not well defined as $\gamma(u)$, is undefined for a non-positive vector $u \leqslant 0$. The following relationship holds:
\[\gamma(u) \leqslant \|u\|_{\infty} \leqslant \|u\|_2\]
For the generalization term, we considered only $l_{\infty}$ and $l_2$ norm as the first function is undefined for some cases, since there is no guarantee that $X_k^{-1}\delta(x_k) \geqslant 0$ will always hold for all $k \geqslant 1$. For our analysis, we select long-step size and long generalization parameter, i.e., we redefine the update formula \eqref{eq:3} for $k \geqslant 1$ as follows:
\begin{align}
\label{eq:5}
x_{k+1} \overset{\underset{\mathrm{def}}{}}{=} x_k - \alpha \frac{X_k^2s_k}{\gamma(X_ks_k)} + \beta \frac{\delta(x_k)}{\|X_k^{-1}\delta(x_k)\|_{\infty}}; \ \ x_1 \overset{\underset{\mathrm{def}}{}}{=} x_0- \alpha \frac{X_0^2s_0}{\gamma(X_0s_0)}
\end{align}
Now, by defining $\alpha_k \overset{\underset{\mathrm{def}}{}}{=} \frac{\alpha}{\gamma(X_ks_k)}, \ \beta_k \overset{\underset{\mathrm{def}}{}}{=} \frac{\beta}{\|X_k^{-1}\delta(x_k)\|_{\infty}}$, we get the modified update formula as follows:
\begin{align}
\label{eq:6}
x_{k+1} = x_k - \alpha_k X_k^2s_k + \beta_k \delta(x_k); \ \ x_1 = x_0- \alpha_0 X_0^2s_0
\end{align}
Let us assume $\lim_{k \to \infty} x_k = x^*$, then define sequences $\{u_k\}, \{\gamma_k\},  \{r_k\}$ and  $\{p_k\}$ as follows:
\begin{align}
\label{eq:7}
u_k \overset{\underset{\mathrm{def}}{}}{=} \frac{X_ks_k}{c^Tx_k-c^Tx^*} , \ \ \gamma_k \overset{\underset{\mathrm{def}}{}}{=} \prod_{j=1}^{k} \beta_k, \ \ r_k \overset{\underset{\mathrm{def}}{}}{=} \frac{\gamma_k}{c^Tx_k-c^Tx^*}, \ \ p_k \overset{\underset{\mathrm{def}}{}}{=} \gamma_k X_k^{-1}e
\end{align}

\begin{theorem}
\label{theorem-1} 
The sequences $\{x_{k+1}\}$ and $ \{x_k\}$ generated by the GAFS algorithm satisfy the following two identities for all $k \geqslant 0$:
\begin{align*}
& \frac{c^Tx_{k+1}- c^Tx^*}{c^Tx_{k}- c^Tx^*} = 1- \alpha \sum\limits_{j=0}^{k} \frac{\|u_j\|^2}{\gamma(u_j)}\frac{r_k}{r_j} \\
 & \left(X_k^{-1}x_{k+1}\right)_j \ = \   \frac{(x_{k+1})_j}{(x_{k})_j} \ = \ 1- \alpha \sum\limits_{i=0}^{k} \frac{(u_i)_j}{\gamma(u_i)}\frac{(p_k)_j}{(p_i)_j} 
\end{align*}
\end{theorem}

\begin{proof}
Taking inner product with $c$ in both sides of \eqref{eq:6} and using the definitions from \eqref{eq:7}, we can find the following relationship:
\begin{align}
c^Tx_{k}-c^Tx_{k-1} & = -\alpha_{k-1} \|X_{k-1}s_{k-1}\|^2 + \beta_{k-1} (c^Tx_{k-1}-c^Tx_{k-2}) \nonumber \\
& = -\alpha_{k-1} \|X_{k-1}s_{k-1}\|^2 - \beta_{k-1} \alpha_{k-2} \|X_{k-2}s_{k-2}\|^2 \nonumber \\  
& \ \ \ + \beta_{k-1} \beta_{k-2} (c^Tx_{k-2}-c^Tx_{k-3}) \nonumber \\
& \vdots \nonumber \\
& = \gamma_{k-1} (c^Tx_1-c^Tx_0) - \sum\limits_{j=1}^{k-1} \frac{\gamma_{k-1}}{\gamma_j} \alpha_j \|X_js_j\|^2  \nonumber \\
& = -\alpha_0 \gamma_{k-1} \|X_0s_0\|^2 - \sum\limits_{j=1}^{k-1} \frac{\gamma_{k-1}}{\gamma_j} \alpha_j \|X_js_j\|^2 \nonumber \\
& =  - \alpha \sum\limits_{j=0}^{k-1} \frac{\|u_j\|^2}{\gamma(u_j)}\frac{\gamma_{k-1}}{\gamma_j} (c^Tx_j-c^Tx^*) \label{eq:10}
\end{align}
Now using the update formula \eqref{eq:6}, the definition \eqref{eq:7} and equation \eqref{eq:10}, we find the following equation,
\begin{align}
\frac{c^Tx_{k+1}- c^Tx^*}{c^Tx_{k}- c^Tx^*} & = 1 - \alpha_k \ \frac{c^TX_k^2s_k}{c^Tx_k-c^Tx^*} - \beta_k \frac{c^Tx_{k-1}-c^Tx_{k}}{c^Tx_{k}- c^Tx^*} \nonumber \\
& = 1 -\alpha \frac{\|u_k\|^2}{\gamma(u_k)} - \alpha \sum\limits_{j=0}^{k-1} \frac{\|u_j\|^2}{\gamma(u_j)}\frac{\beta_k\gamma_{k-1}}{\gamma_j} \frac{c^Tx_j-c^Tx^*}{c^Tx_k-c^Tx^*} \\
& = 1- \alpha \sum\limits_{j=0}^{k} \frac{\|u_j\|^2}{\gamma(u_j)}\frac{r_k}{r_j} \label{eq:11}
\end{align}
The above equation \eqref{eq:11} proves part (1) of Theorem \ref{theorem-1}. Similarly, using equation \eqref{eq:6} and \eqref{eq:7}, we have,
\begin{align}
x_{k+1}-x_k & = -\alpha_k X_k^2s_k + \beta_k \delta(x_k) \nonumber\\
& = -\alpha_k X_k^2s_k - \beta_k \alpha_{k-1} X_{k-1}^2s_{k-1} + \beta_k \beta_{k-1} \delta(x_{k-1}) \nonumber \\
& \vdots \nonumber \\
& = \gamma_k \delta(x_1) - \sum\limits_{j=1}^{k} \frac{\gamma_k}{\gamma_j} \alpha_j X_j^2s_j \ = - \alpha \sum\limits_{j=0}^{k}\frac{\gamma_k}{\gamma_j} \frac{X_ju_j}{\gamma(u_j)} \label{eq:12}
\end{align}
Then, multiplying both sides of \eqref{eq:12} by $X_k^{-1}$ and after simplification, we have for all $j =1,2,...,n.$,
\begin{align}
\left(X_k^{-1}x_{k+1}\right)_j \ = \   \frac{(x_{k+1})_j}{(x_{k})_j} \ &= \ 1 - \alpha \sum\limits_{i =0}^{k}\frac{(u_i)_j}{\gamma(u_i)}\frac{\gamma_k(x_i)_j}{\gamma_i (x_k)_j} \nonumber \\
& = \ 1- \alpha \sum\limits_{i=0}^{k} \frac{(u_i)_j}{\gamma(u_i)}\frac{(p_k)_j}{(p_i)_j} \label{eq:13}
\end{align}
The above equations \eqref{eq:11} and \eqref{eq:13} prove part (1) and part (2) of Theorem \ref{theorem-1}, respectively.
\end{proof}
\noindent Now, for the rest of our analysis let us define the set $Q$ as below:
\[Q \overset{\underset{\mathrm{def}}{}}{=} \{(\alpha, \beta) | \ 0 < \alpha < 1 ,\  0 \leqslant \beta < \frac{1}{\phi}, \ \alpha +\beta \leqslant \frac{2}{3} \}\]
Where $\phi = 1.618...$ is the so called golden ratio.
\begin{theorem}
\label{theorem-2}
For $\alpha, \beta \in Q$, starting from a strictly feasible point $x_0$, the sequence $x_{k+1}$ generated by the update formula \eqref{eq:6} has the following three properties for all $k \geqslant 0$:
\begin{enumerate}
\item $Ax_{k+1} =b$ 
\item $x_{k+1} > 0$ 
\item $c^Tx_{k+1} < c^Tx_k$ 
\end{enumerate}
\end{theorem}

\begin{proof}
Since the sequence $\bar{v_j} = \frac{X_j^2s_j}{\gamma(X_js_j)}$ solves the EAP problem \eqref{eq:2} for all $j \geqslant 0$, we have $A\bar{v_j} = 0$ for all $j \geqslant 0$. As $Ax_0 = b$, using equation \eqref{eq:12}, we have,
\begin{align*}
Ax_{k+1} & = Ax_0 + \sum\limits_{l=0}^{k} A \delta(x_{l+1}) \\
& = Ax_0 - \alpha \sum\limits_{l =0}^{k}\sum\limits_{j=0}^{l}\frac{\gamma_l}{\gamma_j} \frac{AX_ju_j}{\gamma(u_j)} \\
& = b - \alpha \sum\limits_{l =0}^{k}\sum\limits_{j=0}^{l}\frac{\gamma_l}{\gamma_j} \frac{AX_j^2s_j}{\gamma(X_js_j)} = b - \alpha *0 = b
\end{align*}
This proves part (1) of Theorem \ref{theorem-2}. For the second part, let us evaluate the upper bound of $\|X_k^{-1}\delta(x_{k+1})\|_{\infty}$,
\begin{align*}
\|X_k^{-1}\delta(x_{k+1})\|_{\infty} & = \Big\| - \alpha \frac{X_ks_k}{\gamma(X_ks_k)} +\beta  \frac{X_k^{-1}\delta(x_k)}{\|X_k^{-1}\delta(x_k)\|_{\infty}} \Big\|_{\infty} \\
& \leqslant \ \alpha \ \frac{\|X_ks_k\|_{\infty}}{\gamma(X_ks_k)} + \beta \ \frac{ \|X_k^{-1}\delta(x_k)\|_{\infty}}{\|X_k^{-1}\delta(x_k)\|_{\infty}} \\
& \leqslant \ \alpha + \beta \ \leqslant \ \frac{2}{3} \ < \ 1
\end{align*}
In particular, for all $j = 1,2,...,n$, we have,
\begin{align}
\label{eq:14}
\frac{|x_{k+1}^{j}-x_{k}^j|}{x_k^{j}} \ \leqslant \ \|X_k^{-1}\delta(x_{k+1})\|_{\infty} \ \leqslant \ \alpha + \beta \ \leqslant \ \frac{2}{3} \ < \ 1
\end{align}
Which implies $x_{k+1}^j > 0$ for all $j$. Therefore, $x_{k+1} > 0$ for all $k \geqslant 0$. Now equation using \eqref{eq:5} we have,
\begin{align}
c^Tx_{k+1} & = c^Tx_k- \alpha \frac{c^TX_k^2s_k}{\gamma(X_ks_k)} + \beta  \frac{c^T\delta(x_k)}{\|X_k^{-1}\delta(x_k)\|_{\infty}} \nonumber \\
& \leqslant \  c^Tx_k- \alpha \|X_ks_k\|+ \beta  \frac{c^T\delta(x_k)}{\|X_k^{-1}\delta(x_k)\|_{\infty}} \nonumber \\
& < \ c^Tx_k +\beta_k \left(c^Tx_k-c^Tx_{k-1}\right) \nonumber \\
& < \ c^Tx_k + \gamma_k \left(c^Tx_1-c^Tx_0\right) \ = \ c^Tx_k - \alpha \gamma_k \frac{\|X_0s_0\|^2}{\gamma(X_0s_0)}  \ < \ c^Tx_k \label{eq:15}
\end{align}
Therefore, $c^Tx_{k+1} < c^Tx_k$ for all $k \geqslant 0$.
\end{proof}

\begin{theorem}
\label{theorem-3}
The following statements hold for the GAFS algorithm:
\begin{enumerate}
\item The sequence of objective function values $\{c^Tx_k\}$ generated by the GAFS strictly decreases and converges to a finite value.
\item $X_ks_k \rightarrow \textbf{0}$ as $k \rightarrow \infty$.
\end{enumerate}
\end{theorem}

\begin{proof}
As a consequence of Theorem \ref{theorem-2}, we know that the sequence $\{c^Tx_k\}$ is a decreasing sequence. For the part (1) of Theorem \ref{theorem-3}, we just need to show that the sequence $\{c^Tx_k\}$ is bounded. As per our assumption, $x^*$ is the optimal solution of the primal problem (P) in \eqref{eq:1}. However, it implies the following,
\[ c^Tx^* \leqslant \dots < c^Tx_{k+1} < c^Tx_k < ...< c^Tx_0\] 
It means that the sequence $\{c^Tx_k\}$ is bounded. Therefore, using the \textit{Monotone Convergence Theorem} we can conclude that the sequence $\{c^Tx_k\}$ is convergent and $\lim_{k \rightarrow \infty} c^Tx_k = c^Tx^*$. For the second part, we see that,
\begin{align}
0 \ < \ \|X_ks_k\| & \ \leqslant \ \frac{1}{\alpha} \left[(c^Tx_{k}-c^Tx_{k+1})+ \beta_k (c^Tx_k-c^Tx_{k-1})\right] \nonumber \\
& < \ \frac{1}{\alpha}  \left[(c^Tx_{k}-c^Tx_{k+1})+ \gamma_k (c^Tx_1-c^Tx_{0})\right] \label{eq:16}
\end{align}
Now, by the properties of $\{c^Tx_k\}$, we have $ c^Tx_0 -c^Tx^* < \infty $ and $\bar{c} =  c^Tx_0-c^Tx_1 < \infty $, also as a consequence of Lemma \ref{lemma-6}, we have $G = \sum\limits_{k =0}^{\infty} \gamma_k < \infty$. Combining these facts and equation \eqref{eq:16}, we can write the following equation,
\begin{align}
 \sum\limits_{k =0}^{\infty}\|X_ks_k\| \ & < \ \frac{1}{\alpha} \left[\sum\limits_{k =0}^{\infty} (c^Tx_{k}-c^Tx_{k+1}) + \bar{c} \sum\limits_{k =0}^{\infty} \gamma_k \right] \nonumber \\
 & = \ \frac{1}{\alpha} \left[c^Tx_0 -c^Tx^* + \bar{c} \ G\right] \ < \ \infty \label{eq:17}
\end{align}
Now, equation \eqref{eq:17} allows us to write $ X_ks_k \rightarrow \textbf{0} \ \text{as} \ k \rightarrow \infty $. This proves second part of Theorem \ref{theorem-3}.
\end{proof}

\begin{remark}
\label{rem-2}
By using Theorem \ref{theorem-3}, we see that the complementary slackness condition holds in the limit since, $\lim_{k \rightarrow \infty} (x_k)_j(s_k)_j = 0$ for all $j = 1,2,3,...,n$.
\end{remark}

\begin{theorem}
\label{theorem-4}
The following statements hold for the GAFS algorithm:
\begin{enumerate}
\item The sequence $\{x_k\}$ converges to a point $x^*$, belongs to interior of the primal feasible region.
\item For all $k \geqslant 0$ there exists a $N = N(x,A) >0$ such that,
\[\|x_k-x^*\| \ \leqslant \ M \left(c^Tx_k-c^Tx^*\right) +\frac{\|x_{1}-x_0\| }{\gamma_1} G(k) \ \leqslant \ (M+N) \left(c^Tx_k-c^Tx^*\right) \]
\end{enumerate}
\end{theorem}

\begin{proof} Denoting $t_k = (c^Tx_{k}-c^Tx_{k-1})$, as a direct consequence of equation \eqref{eq:15} and Lemma \ref{lemma-4} we have,
\begin{align}
\|x_{k+1}-x_k\| \ & \leqslant \ \alpha M \frac{c^Td_k}{\|X_ks_k\|}+\frac{\beta}{\|X_k^{-1}\delta(x_k)\|_{\infty}} \|x_{k}-x_{k+1}\| \nonumber \\
& =  M \left(c^Tx_{k}-c^Tx_{k+1}\right) + \frac{M \beta \left(c^Tx_{k}-c^Tx_{k-1}\right)}{\|X_k^{-1}\delta(x_k)\|_{\infty}} + \frac{\beta \|x_{k}-x_{k+1}\|}{\|X_k^{-1}\delta(x_k)\|_{\infty}} \nonumber \\
& = M \beta_k t_k - M t_{k+1}+ \beta_k \delta(x_{k+1}) \nonumber \\
& = M \sum \limits _{j= 1}^{k-1} \frac{\gamma_{k}}{\gamma_j} t_{j+1}- M \sum \limits _{j= 2}^{k} \frac{\gamma_{k}}{\gamma_j} t_{j+1} + \frac{\gamma_k}{\gamma_1} \delta(x_{1}) \nonumber \\
& = M \frac{\gamma_k}{\gamma_1} t_2 - M \frac{\gamma_k}{\gamma_k} t_{k+1} + \frac{\gamma_k}{\gamma_1} \delta(x_{1}) \nonumber \\
& = M \frac{\gamma_k}{\gamma_1} \left(c^Tx_{2}-c^Tx_{1}\right) + M \left(c^Tx_{k}-c^Tx_{k+1}\right) + \frac{\gamma_k}{\gamma_1}  \|x_1-x_0\| \label{eq:19}
\end{align}
Furthermore,  from Lemma \ref{lemma-6}, we know that the sequence $\gamma_k$ converges to $0$ as $k \rightarrow \infty$, so we can assume that the sequence $\left\{\sum\limits_{k=1}^{m} \gamma_k \right\} $ converges to some finite value $G$ as $m$ goes to infinity, i.e.,
\[ \sum\limits_{k=1}^{\infty} \gamma_k = G < \infty\]
Then from equation \eqref{eq:19} we have,
\begin{align*}
\sum\limits_{k =0}^{\infty} & \|x_{k+1}-x_k\|  = \|x_{1}-x_0\| + \sum\limits_{k =1}^{\infty} \|x_{k+1}-x_k\| \\
& \leqslant \|x_{1}-x_0\| + M \sum\limits_{k =1}^{\infty} \left(c^Tx_{k}-c^Tx_{k+1}\right) + \frac{M}{\gamma_1}\left(c^Tx_{2}-c^Tx_{1}\right) \\
& \ \ + \frac{\|x_{1}-x_0\|}{\gamma_1} \sum\limits_{k =1}^{\infty} \gamma_k \\
& = \|x_{1}-x_0\| + M \left(c^Tx_{1}-c^Tx^*\right) + \frac{G}{\gamma_1} \left[M \left(c^Tx_{2}-c^Tx_{1}\right) + \|x_{1}-x_0\| \right] \\
& < \infty
\end{align*}
The above identity shows that, $\{x_k\}$ is a Cauchy sequence, and therefore, it is a convergence sequence (i.e., every real Cauchy sequence is convergent). Now, for all $0 \leqslant k \leqslant l$, using equation \eqref{eq:19} we have,
\begin{align}
\|x_l-x_k\| & \ \leqslant \ \big \|\sum \limits_{j =k}^{l-1} \left(x_{j+1}-x_{j}\right)\big \| \ \leqslant \ \sum \limits_{j =k}^{l-1} \|x_{j+1}-x_{j}\| \nonumber \\
& \leqslant  M \sum \limits _{j=k}^{l-1} \left(c^Tx_j-c^Tx_{k+1}\right)+ \frac{1}{\gamma_1} \left[M \left(c^Tx_{2}-c^Tx_{1}\right) + \|x_{1}-x_0\| \right]\sum\limits_{j=k}^{l-1} \gamma_j \nonumber \\
& \ \leqslant \ M \sum \limits _{j=k}^{l-1} \left(c^Tx_j-c^Tx_{k+1}\right)+  \frac{\|x_{1}-x_0\| }{\gamma_1}\sum\limits_{j=k}^{l-1} \gamma_j \label{eq:20}
\end{align}
Now, letting $l \rightarrow \infty$ in \eqref{eq:20} and defining $G(k) \overset{\underset{\mathrm{def}}{}}{=} \sum\limits_{j = k}^{\infty} \gamma_j $, we have,
\begin{align}
\|x_k-x^*\| \ &\leqslant \ M \left(c^Tx_k-v^*\right) +\frac{\|x_{1}-x_0\| }{\gamma_1} G(k) \nonumber\\ 
& \leqslant M \left(c^Tx_k-c^Tx^*\right) + \frac{\|x_{1}-x_0\| }{\gamma_1} \bar{N} \left(c^Tx_k-c^Tx^*\right) \nonumber \\
 & = (M+ N) \left(c^Tx_k-c^Tx^*\right) \label{eq:21}
\end{align}
This is the required bound. In the last line, we used Lemma \ref{lemma-8} with $N = \frac{\|x_{1}-x_0\| }{\gamma_1} \bar{N}$.
\end{proof}

\begin{theorem}
\label{theorem-x}
The sequence $\{B(x_k)\}$ generated by the AAFS algorithm converges to the same point $x^*$ and belongs to the interior of the primal feasible region.
\end{theorem}
\begin{proof}
From Theorem \ref{theorem-4}, we know that the sequence $\{x_k\}$ generated by GAFS converges to $x^*$. Then using the definition \eqref{eq:4} and the basic idea of SST, we can immediately conclude that for all $j =1,2,...,n$, the following relation holds:
\begin{align*}
\lim_{k \to \infty} (B(x_k))_j = (x^*)_j
\end{align*}
Since, this holds for all $j =1,2,...,n$, we can prove $\lim_{k \to \infty} B(x_k) = x^*$. The last part of Theorem \ref{theorem-x} follows from the fact that $(B(x_k))_j > 0$, for all $k \geqslant 1$ and $j =1,2,...,n$.
\end{proof}

\noindent With $\lim_{k \rightarrow \infty}x_k = x^*$, let us define the sets $N$ and $B$ as follows,
\[N \overset{\underset{\mathrm{def}}{}}{=} \{i \ | \ x_i^* = 0\}, \ \ B = \{i \ | \ x_i^* > 0\}, \ \ |N|= p \]
Now, we provide proof for an important property of the sequence $\{x_k\}$, which  subsequently holds for original AFS algorithm. We showed that it holds for GAFS too with different constant. For the original AFS algorithm, the Theorem was proven by several authors in their work \cite{saigal:1996} and \cite{tsuchiya:1995}.
\begin{theorem}
\label{theorem-5}
There exists a $\delta > 0$ and a $R > 0$ such that for each $k \geqslant 0$
\[\frac{c^Tx_k-c^Tx^*}{\|x_k-x^*\|} \ \geqslant \ \frac{1}{R}, \quad  \frac{c^Tx_k-c^Tx^*}{\sum_{i \in N}(x_k)_i} \geqslant \delta, \quad \frac{c^Tx_k-c^Tx^*}{\sum_{i \in B}\big | (x_k)_i-(x^*)_i \big |} \geqslant \delta\]
\end{theorem}
\begin{proof}
Let, $R = M+N$, then from equation \eqref{eq:21}, we have for all $k \geqslant 0$,
\begin{align*}
\frac{c^Tx_k-c^Tx^*}{\|x_k-x^*\|} \ \geqslant \ \frac{1}{M+N} \ = \ \frac{1}{R}
\end{align*}
It proves the first part of Theorem \ref{theorem-5}. Similarly from equation \eqref{eq:21},
\begin{align*}
& c^Tx_k-c^Tx^* \ \geqslant \frac{\|x_k-x^*\|}{R} \geqslant \ \frac{\|x_{k,N}\|}{R} \geqslant \  \frac{\sum_{i \in N}(x_k)_i}{\sqrt{p} R} \\
& c^Tx_k-c^Tx^* \ \geqslant \frac{\|x_k-x^*\|}{R} \geqslant \ \frac{\|x_{k,B}-x_B^*\|}{R} \geqslant \   \frac{\sum_{i \in B}\big | (x_k)_i-(x^*)_i \big |}{\sqrt{n-p} R} 
\end{align*}
By denoting  $\delta = \min \{\frac{1}{\sqrt{p}R}, \frac{1}{\sqrt{n-p}R}\}$, we have the remaining results of Theorem \ref{theorem-5}.
\end{proof}

\begin{theorem}
If $\alpha, \ \beta \in Q$, then the following identities hold:
\label{theorem-6}
\begin{enumerate}
\item For all $k \geqslant 1$,
\[\|X_k^{-1}X_{k-1}\|_\infty  \ < \ 5 \]
\item There exists a $L_2 \geqslant 1$ such that for all $k \geq L_2$,
\begin{align*}
-\frac{\|X_ks_k\|}{c^Tx_k-c^Tx^*} \leqslant \frac{-1}{\|X_k^{-1}\left(x_k-x^*\right)\|} \leqslant \frac{-1}{\sqrt{n}} 
\end{align*}
\item For all $k \geqslant 1$,
\begin{align*}
\frac{1}{\gamma_k} = \ \frac{1}{\beta_1 \beta_2 ... \beta_k} <  \ \frac{\prod_{j=1}^{k} \|X_j^{-1}X_{j-1}\|}{\beta^k} (\alpha +\beta)^k \ < \ \left(\frac{5}{\beta}\right)^k
\end{align*}
\end{enumerate}
\end{theorem}

\begin{proof}
From equation \eqref{eq:14} for all $j = 1,2,...,n$, we have,
\begin{align}
\label{eq:24}
\frac{|(x_{k+1}-x_{k})_j|}{(x_k)_j} \ \leqslant \ \|X_k^{-1}\delta(x_{k+1})\|_{\infty} \ \leqslant \ \alpha + \beta \ \leqslant \ \frac{2}{3}
\end{align}
Simplifying equation \eqref{eq:24} further for all $k \geqslant 1, \ j = 1,2,...,n$ we have,
\begin{align}
\label{eq:25}
\frac{3}{5} \ \leqslant \ \frac{(x_{k-1})_j}{(x_k)_j} \ \leqslant \ 3
\end{align}
Then using equation \eqref{eq:25} to the definition of maximum norm, we have,
\begin{align*}
\|X_k^{-1}\delta(x_k)\|_{\infty} \ \leqslant \ \max_j \left\{1+\frac{(x_{k-1})_j}{(x_k)_j}\right\} \ \leqslant  \ 4
\end{align*}
Therefore, we have,
\begin{align}
\label{eq:26}
\|X_k^{-1}x_{k-1}\|_{\infty} \ \leqslant \ \|e - X_k^{-1}x_{k-1}\|_{\infty} + \|e\|_{\infty} \ \leqslant \ 4  + 1 = 5 
\end{align}
Which proves part (1) of Theorem \ref{theorem-6}. Part (2) of Theorem \ref{theorem-6} is well studied in the literature ( see \cite{saigal:1996}, \cite{tsuchiya:1995}). We can prove part (2) of this Theorem easily as the sequence $\frac{X_k^2s_k}{\|X_ks_k\|}$ generated by the GAFS algorithm solves the EAP problem provided in equation \eqref{eq:2}, i.e., there exists a $L_2$ such that for all $k > L_2$,
\begin{align*}
-\frac{\|X_ks_k\|}{c^Tx_k-c^Tx^*} \ \leqslant \ \frac{-1}{\|X_k^{-1}\left(x_k-x^*\right)\|} \ \leqslant \ \frac{-1}{\sqrt{n}} 
\end{align*}
For proving the last part, we first need an upper bound of $\frac{1}{\beta_k}$. For all $k \geqslant 1$, we have,
\begin{align*}
\frac{1}{\beta_{k}} = \frac{\|X_{k}^{-1}\delta(x_k)\|_{\infty}}{\beta} &= \big \|\frac{X_{k}^{-1}\delta(x_{k-1})}{\|X_{k-1}^{-1}\delta(x_{k-1})\|_{\infty}} - \frac{\alpha}{\beta} \frac{X_{k}^{-1}X_{k-1}^2s_{k-1}}{\gamma(X_{k-1}s_{k-1})}\big \|_{\infty}  \\
& \leqslant \|X_{k}^{-1}X_{k-1}\|_{\infty} \big \|\frac{X_{k-1}^{-1}\delta(x_{k-1})}{\|X_{k-1}^{-1}\delta(x_{k-1})\|_{\infty}} - \frac{\alpha}{\beta} \frac{X_{k-1}s_{k-1}}{\gamma(X_{k-1}s_{k-1})}\big \|_{\infty}  \\
& \leqslant \|X_{k}^{-1}X_{k-1}\|_{\infty}\left(1+ \frac{\alpha}{\beta}\right) \  < \ \frac{5(\alpha + \beta)}{ \beta} \ < \ \frac{5 (\beta +2)}{3\beta} \ < \ \frac{5}{\beta} 
\end{align*}
Here, we use the identity in equation \eqref{eq:26}. Then by the definition of $\gamma_k$ for all $k \geqslant 1$ we have,
\begin{align*}
\frac{1}{\gamma_k} = \ \frac{1}{\beta_1 \beta_2 ... \beta_k} \  < \ \left(\frac{5}{\beta}\right)^k
\end{align*}
It proves the remaining parts of Theorem \ref{theorem-6}.
\end{proof}

\noindent For the remaining sections, let us define sequences $\{u_k\}, \{v_k\}$ and $\{h_k\}$ as follows:
\begin{align}
\label{eq:30}
u_k \overset{\underset{\mathrm{def}}{}}{=} \frac{X_ks_k}{c^Tx_k-c^Tx^*}, \ \  v_k \overset{\underset{\mathrm{def}}{}}{=} \frac{X_k^{-1}\delta(x_k)}{\|X_k^{-1}\delta(x_k)\|_{\infty}}(c^Tx_k-c^Tx^*) \ \ h_k \overset{\underset{\mathrm{def}}{}}{=} \frac{c^T\delta(x_k)}{\|X_k^{-1}\delta(x_k)\|_{\infty}}
\end{align}

\section{Convergence Rate}
\label{sec:rate}
In this section, we measured the significance of AAFS over GAFS in terms of convergence rate given that GAFS converges linearly (see the following Theorem \ref{theorem-7}). The following Theorem gives us the linear convergence rate of GAFS algorithm.

\begin{theorem}
\label{theorem-7}
The following statements holds for the GAFS algorithm:

\begin{enumerate}
\item There exists a $L \geqslant 1$ such that for all $k \geqslant L$,
\begin{align*}
 \frac{c^Tx_{k+1}- c^Tx^*}{c^Tx_{k}- c^Tx^*} \leqslant 1- \frac{\alpha}{\sqrt{n}}- \frac{\alpha}{\sqrt{n}} \left(\frac{\beta}{5}\right)^k
\end{align*}
\item For $\alpha, \beta \in Q $ the following limit holds:
\begin{align*}
 \lim_{k \to \infty}\frac{c^Tx_{k+1}- c^Tx^*}{c^Tx_{k}- c^Tx^*} = 1- \frac{\alpha}{\sqrt{n}} \ < \ 1
\end{align*}
\end{enumerate}
 
\end{theorem}

\begin{proof}
We used Theorem \ref{theorem-1} for proving part (1) of this Theorem. First, let us choose $L = L_2$ (part (2) of Theorem \ref{theorem-6}), then using the update formula \eqref{eq:5} for all $k > L$, we have, 
\begin{align*}
\frac{c^Tx_{k+1}- c^Tx^*}{c^Tx_{k}- c^Tx^*} & = 1- \frac{\alpha }{\gamma(X_ks_k)} \frac{c^TX_k^2s_k}{(c^Tx_{k}- c^Tx^*)} + \frac{\beta}{\|X_k^{-1}\delta(x_k)\|_{\infty}} \frac{c^Tx_{k}- c^Tx_{k-1}}{(c^Tx_{k}- c^Tx^*)}  \\
& \leqslant 1- \alpha \ \frac{\|X_ks_k\|}{c^Tx_{k}- c^Tx^*} + \beta_k \ \frac{c^Tx_{k}- c^Tx_{k-1}}{c^Tx_{k}- c^Tx^*}  \\
& \leqslant 1- \frac{\alpha}{\|X_k^{-1}\left(x_k-x^*\right)\|} + \beta_k \gamma_{k-1} \ \frac{c^Tx_{1}- c^Tx_{0}}{c^Tx_{k}- c^Tx^*}  \\
& \leqslant 1- \frac{\alpha}{\sqrt{n}} - \alpha \gamma_k \ \frac{\|X_0s_0\|}{c^Tx_{k}- c^Tx^*}  \\
& \leqslant 1- \frac{\alpha}{\sqrt{n}} - \alpha \gamma_k \ \frac{\|X_ks_k\|}{c^Tx_{k}- c^Tx^*}  \\
& \leqslant 1 - \frac{\alpha}{\sqrt{n}}- \frac{\alpha \gamma_k}{\sqrt{n}} \ \leqslant \  1- \frac{\alpha}{\sqrt{n}}- \frac{\alpha}{\sqrt{n}} \left(\frac{\beta}{5}\right)^k \ < \ 1- \frac{\alpha}{\sqrt{n}}
\end{align*}
 Here, we used the the fact that the sequence $\{\|X_ks_k\|\}$ is a decreasing sequence and converges to zero due to the property of complementary slackness, i.e., $\|X_0s_0\| \ \geqslant \ \|X_1s_1\| \ \geqslant ... \geqslant \  \|X_ks_k\|$ (see part (2) of Theorem \ref{theorem-3}).

Now, part (2) of Theorem \ref{theorem-7} is a direct consequence of part (1) of Theorem \ref{theorem-7} as the sequence $\{\left(\frac{\beta}{4 \sqrt{n}}\right)^k\}$ converges to zero as $k \rightarrow \infty$.
\end{proof}

\begin{remark}
Note that, Theorem \ref{theorem-7} indicates that GAFS algorithm converges linearly. Next, we compared the convergence rates between the proposed algorithms. In other words, how good is the sequence $\{c^TB(x_k)\}$ compared to the sequence $\{c^Tx_k\}$ when the latter converges linearly.  The next Theorem (Theorem \ref{theorem-ac1}) shows that it is better, in the sense that it converges faster, meaning that the AAFS accelerates the convergence of GAFS.
\end{remark}

\begin{theorem}
\label{theorem-ac1}
Assume that the sequence $\{c^Tx_k\}$ converges to $c^Tx^*$ linearly. Let $B(x_k)$ be as in \eqref{eq:4}, then the sequence $\{c^TB(x_k)\}$, converges faster to $c^Tx^*$ than $\{c^Tx_k\}$ in the sense that,
\begin{align*}
\lim_{k \rightarrow \infty} \frac{c^TB(x_k)-c^Tx^*}{c^Tx_k-c^Tx^*} = 0
\end{align*}
\end{theorem}

\begin{proof} By virtue of Theorem \ref{theorem-7}, we can argue that there exists sequences $\{\sigma_j\}$ and $\{(\lambda_k)_j\}$ such that,
\begin{align}
\label{eq:31}
 \frac{c_j(x_{k+1})_j- c_j(x^*)_j}{c_j(x_{k})_j- c_j(x^*)_j} = \sigma_j + (\lambda_k)_j \ \ \forall \ j = 1,2,...,n, \ k \geq 1 
\end{align}
And $\lim_{k \rightarrow \infty} (\lambda_k)_j = 0 \ \forall \ j$. Simplifying \eqref{eq:31} for $(k+1)$th and $(k+2)$th terms, we have,
\begin{align}
 & c_j(x_{k+1})_j = c_j(x^*)_j + \left(\sigma_j + (\lambda_k)_j\right) \left(c_j(x_{k})_j- c_j(x^*)_j\right) \nonumber \\
 & c_j(x_{k+2})_j = c_j(x^*)_j + \left(\sigma_j + (\lambda_{k+1})_j\right) \left(c_j(x_{k+1})_j- c_j(x^*)_j\right) \label{eq:32}
\end{align}
for all $ j = 1,2,...,n, \ k \geq 1$. Now using \eqref{eq:31} and \eqref{eq:32}, we have,
\begin{align}
 \frac{c^TB(x_k)-c^Tx^*}{c^Tx_k-c^Tx^*} & = \frac{c^Tx_k-c^Tx^*}{c^Tx_k-c^Tx^*} \\ 
 & - \sum\limits_{j=1}^{n}  \frac{[c_j(x_k)_j-c_j(x_{k+1})_j]^2}{[c^Tx_k-c^Tx^*][c_j(x_k)_j-2c_j(x_{k+1})_j+c_j(x_{k+2})_j]} \nonumber \\
 = 1 & - \frac{1}{n} \sum\limits_{j=1}^{n}  \frac{[c_j(x_k)_j-c_j(x_{k+1})_j]^2}{[c_j(x_k)_j-c_j(x^*)_j][c_j(x_k)_j-2c_j(x_{k+1})_j+c_j(x_{k+2})_j]} \nonumber \\
  = 1 & - \frac{1}{n} \sum\limits_{j=1}^{n} \frac{\left[\frac{c_j(x_k)_j-c_j(x_{k+1})_j}{c_j(x_k)_j-c_j(x^*)_j}\right]^2}{\frac{c_j(x_k)_j-2c_j(x_{k+1})_j+c_j(x_{k+2})_j}{c_j(x_k)_j-c_j(x^*)_j}} \nonumber \\
  = 1 & - \frac{1}{n} \sum\limits_{j=1}^{n} \frac{\left[\sigma_j + (\lambda_k)_j-1\right]^2}{\left[\sigma_j + (\lambda_k)_j\right]\left[\sigma_j + (\lambda_{k+1})_j\right]- 2 \left[\sigma_j + (\lambda_k)_j\right]+1} \label{eq:33}
\end{align}
Taking the limit $k \rightarrow \infty$ in \eqref{eq:33} and using the property of sequence $\{(\lambda_k)_j\}$, we have,

\begin{multline*}
\lim_{k \rightarrow \infty} \frac{c^TB(x_k)-c^Tx^*}{c^Tx_k-c^Tx^*} \\
= 1- \lim_{k \rightarrow \infty} \frac{1}{n} \sum\limits_{j=1}^{n} \frac{\left[\sigma_j + (\lambda_k)_j-1\right]^2}{\left[\sigma_j + (\lambda_k)_j\right]\left[\sigma_j + (\lambda_{k+1})_j\right]- 2 \left[\sigma_j + (\lambda_k)_j\right]+1}  \\
 = 1- \frac{1}{n} \sum\limits_{j=1}^{n} \frac{(\sigma_j-1)^2}{\sigma_j^2-2 \sigma_j+1} \ = 1- 1 = 0
\end{multline*}
This proves the required Theorem.
\end{proof}

\section{Convergence of the Dual sequence}
\label{sec:dual}
In this section, we introduced a local version of potential function largely studied in the literature. For the convergence of the dual sequence, it is required to control both the step sizes $\alpha$ and $\beta$. For the original Affine Scaling method, it was first shown by Tsuchiya \textit{et al.} \cite{tsuchiya:1995} that for the dual convergence we need to have $\alpha \leqslant \frac{2}{3}$. A simpler version of its proof is also available in Saigal \cite{saigal:1996}. Here, we proved that the dual sequence generated by GAFS method converges if we have $\alpha, \beta \in Q$. We see that the original result of AFS can be found with the choice of $ \beta = 0$ in our proof. At first, we introduced the local potential function (defined in \cite{tsuchiya:1991}). For any $x > 0 $ with $c^Tx- c^Tx^* > 0$ and $N = \{j \ | \ (x^*)_j = 0\}$ with $p = |N|$, let us define the following function:
\begin{align*}
F_N(x) \overset{\underset{\mathrm{def}}{}}{=} p \log (c^Tx-c^Tx^*)- \sum\limits_{j \in N} \log (x)_j
\end{align*}

\begin{theorem}
\label{theorem-9} 
For the sequence $\{x_k\}$ we can show that for all $k \geqslant 0$,
\begin{multline}
F_N(x_{k+1})- F_N(x_k) = p \log (1- \theta \|w_{k,N}\|^2- \theta \sigma_k^2+ \theta  w_{k,N}^T v_{k,N}  \\
- \theta \left(\frac{2 \delta_k}{p} -\frac{2 \omega_k}{p}+ \epsilon_k + \phi h_k\right))
- \sum\limits_{j \in N} \log \left(1- \theta (w_k)_j - (\phi- \theta) (v_k)_j\right) \label{eq:35}
\end{multline}
; where $w_{k,N} = u_{k,N}+ v_{k,N}-\frac{1}{p}e, \ \ \bar{\alpha} = \frac{\alpha}{\gamma(u_{k,N})}, \ \ \bar{\beta} = \frac{\beta}{\|v_{k,N}\|_{\infty}},  \  \ \theta = \frac{p\bar{\alpha}}{p- \bar{\alpha}}, \ \ \phi = \frac{p\bar{\beta}}{p- \bar{\alpha}}$,  \\
$\quad \quad \sum\limits_{k = L}^{\infty} |\epsilon_k| < \infty, \quad \sum\limits_{k = L}^{\infty} |\delta_k| < \infty,  \quad \sum\limits_{k = L}^{\infty} |\omega_k| < \infty$
\end{theorem}

\begin{proof}  Using the update formula \eqref{eq:5} and definition \eqref{eq:30}, we have, 
\begin{align}
\frac{c^Tx_{k+1}- c^Tx^*}{c^Tx_{k}- c^Tx^*} & = 1- \alpha \frac{c^TX_k^2s_k}{\gamma(X_ks_k) (c^Tx_k-c^Tx^*)} -\beta \frac{c^Tx_{k-1}-c^Tx_k}{\|X_k^{-1}\delta(x_k)\|_{\infty} (c^Tx_k-c^Tx^*)} \nonumber \\
& = 1- \alpha\frac{\|u_k\|^2}{\gamma(u_k)} - \beta \frac{h_k}{\|v_k\|_{\infty}} \  = \ 1 - \bar{\alpha} \|u_k\|^2 - \bar{\beta} h_k \label{eq:36}
\end{align}
And also for all $j \in N$, we have,
\begin{align}
 \frac{(x_{k+1})_j}{(x_{k})_j} \ & = \ 1- \alpha \frac{(X_ks_k)_j}{\gamma(X_ks_k)} - \beta \frac{\left[X_k^{-1}\delta(x_k)\right]_j}{\|X_k^{-1}\delta(x_k)\|_{\infty}} \nonumber \\
 & = 1- \alpha\frac{(u_k)_j}{\gamma(u_k)} - \beta \frac{(v_k)_j}{\|v_k\|_{\infty}} \ = \ 1 - \bar{\alpha} (u_k)_j - \bar{\beta} (v_k)_j \label{eq:37}
\end{align}
Now from part 2(a) of Lemma \ref{lemma-9} , there exist a $L \geqslant 1$ such that for all $k \geqslant L$, we have,
\begin{align}
 1 - & \bar{\alpha} \|u_k\|^2 - \bar{\beta} h_k \  = \ 1 - \bar{\alpha} \ \| \ w_{k,N} - v_{k,N} + \frac{1}{p} e \ \|^2 -\bar{\alpha} \epsilon_k - \bar{\beta} h_k \nonumber \\
  = \ & 1- \frac{\bar{\alpha}}{p} - \bar{\alpha} (\|w_{k,N}\|^2 +\|v_{k,N}\|^2) + \frac{2\bar{\alpha}}{p} e^T(v_{k,N}-w_{k,N}) \nonumber\\
  & \qquad \qquad \qquad \qquad \qquad \qquad  + 2 \bar{\alpha}  v_{k,N}^T w_{k,N} - \bar{\alpha} \epsilon_k - \bar{\beta} h_k \nonumber \\
 = \ & \frac{p-\bar{\alpha}}{p} \left[1- \theta \left(\|w_{k,N}\|^2 +\sigma_k^2 + \epsilon_k - 2  v_{k,N}^T w_{k,N} + \frac{2}{p} (\delta_k - \omega_k) \right) - \phi h_k\right] \label{eq:38}
\end{align}
This is a simplification of equation \eqref{eq:36}. Also simplifying equation \eqref{eq:37}, we have,
\begin{align}
1 - \bar{\alpha} (u_k)_j - \bar{\beta} (v_k)_j & = 1- \frac{\bar{\alpha}}{p} - \bar{\alpha} (w_k)_j + \bar{\alpha} (v_k)_j - \bar{\beta} (v_k)_j \nonumber \\
&= \frac{p-\bar{\alpha}}{p} \left[1- \theta (w_k)_j - (\phi- \theta) (v_k)_j \right] \label{eq:39}
\end{align}
Using these two identities from \eqref{eq:38} and \eqref{eq:39}, we have the desired result of Theorem \ref{theorem-9}.
\end{proof}

We know from the problem structure and assumptions that the sequence $\{c^Tx_k\}$ is bounded. In the next Theorem (Theorem \ref{theorem-10}) we show that with $\alpha, \beta \in Q$, the dual sequence converges to the analytic center of the optimal face of dual polytope. As defined by Saigal \cite{saigal:1996}, with $D \overset{\underset{\mathrm{def}}{}}{=} \{(y,s) : \ A_B^Ty = c_B, A_N^Ty+ s_N = c_N, s_B = 0\} $, we define the \textit{Analytic Center Problem} (ACP) of the optimal dual face as the solution of $(y^*,s^*)$ to the following problem:
\begin{align}
 \text{max} \ & \ \sum\limits_{j \in N} \log s_j \nonumber \\
 & (y,s) \in D  \label{eq:40} \\
 & \  s_N > 0 \nonumber
\end{align}

\begin{theorem}
\label{theorem-10} If $\alpha, \beta \in Q$, then there exist vectors $x^*, y^*$ and $s^*$ such that the sequences $\{x_k\}$, $\{y_k\}$ and $\{s_k\}$ generated by the GAFS algorithm converges to $x^*, y^*$ and $s^*$, respectively, i.e.,
\begin{enumerate}
\item $x_k \rightarrow x^*$,
\item $y_k \rightarrow y^*$,
\item $s_k \rightarrow s^*$
\end{enumerate}
Where  $x^*, y^*$ and $s^*$ are the optimal solutions of the respective primal and dual problems, and they also satisfy the strict complementary slackness property. Furthermore, the dual pair $(y^*,s^*)$ converges to the analytic center of the optimal dual face and the primal solution $x^*$ converges to the relative interior of the optimal primal face.
\end{theorem}

\begin{proof}
Since, $ \log(1-a) < -a$, we can find a $L_1\geqslant 1$ such that for all $k \geqslant L_1$,
\begin{multline}
F_N(x_{k+1})- F_N(x_k) \  \leqslant  \ -p \theta \left[\|w_{k,N}\|^2 +\sigma_k^2 + \epsilon_k - 2 v_{k,N}^T w_{k,N} + \frac{2}{p} (\delta_k - \omega_k)\right] \\
- p \phi h_k - \sum\limits_{j \in N} \log \left(1- \theta (w_k)_j - (\phi- \theta) (v_k)_j\right) \label{eq:41}
\end{multline}
Now, we analyze equation \eqref{eq:41} for two cases based on the sign of $ \theta \gamma(w_{k,N}) + (\phi-\theta)\gamma(v_{k,N})$. \\
\textbf{Case 1: $\ \ \theta \gamma(w_{k,N}) + (\phi-\theta)\gamma(v_{k,N}) \ \leqslant \ 0 $}\\
Then we must have $\theta (w_k)_j + (\phi- \theta) (v_k)_j \leqslant 0$ for all $j \in N$,  which implies,
\begin{align}
\log \left(1- \theta (w_k)_j - (\phi- \theta) (v_k)_j\right) \ \geqslant \ 0 \quad \text{for all} \quad j \in N \label{eq:42}
\end{align}
Using part (d) of Theorem \ref{theorem-8} and equation \eqref{eq:41} and \eqref{eq:42}, for all $k \geqslant L_1$ we have,
\begin{align}
F_N(x_{k+1}) & - F_N(x_k)  \nonumber \\ 
& \leqslant -p \theta \|w_{k,N}\|^2 + 2p \theta  v_{k,N}^T w_{k,N} - \theta (2 \delta_k + p \epsilon_k- 2\omega_k+ p \sigma_k^2) - p \phi h_k \nonumber \\
& \leqslant  -p \theta \|w_{k,N}\|^2 + 2p \theta \epsilon \|w_{k,N}\|- \theta (2 \delta_k + p \epsilon_k- 2\omega_k+ p \sigma_k^2) - p \phi h_k \label{eq:43}
\end{align}
\textbf{Case 2: $ \ \ \theta \gamma(w_{k,N}) + (\phi-\theta)\gamma(v_{k,N}) \ > \ 0 $}\\ 
Let $\bar{\epsilon} > 0$,  since $\gamma(v_{k,N}) = (c^Tx_k - c^Tx^*) \rightarrow 0$, there exist a $L_2 \geqslant 1$ such that for all $k \geqslant L_2$,
\begin{align*}
\gamma(v_{k,N}) < \bar{\epsilon}
\end{align*}
Then, using the condition of Case 2, we have for all $k \geqslant L_2$,
\begin{align}
\gamma(w_{k,N}) - \gamma(v_{k,N}) & \ > \ (1-\frac{\phi}{\theta}) \gamma(v_{k,N}) - \gamma(v_{k,N}) \ = \ - \frac{\phi}{\theta} \gamma(v_{k,N}) \ > \ - \frac{\phi}{\theta} \bar{\epsilon} \label{eq:44}
\end{align}
Since in equation \eqref{eq:44}, our choice of $\bar{\epsilon} > 0$ is arbitrary, this is true for any $\bar{\epsilon} > 0$, which implies for all $k \geqslant L_2$, we must have $\gamma(w_{k,N}) - \gamma(v_{k,N}) > 0$. Then, from the definition, we have,
\begin{align}
\label{eq:45}
\gamma(u_{k,N}) = \gamma(w_{k,N}- v_{k,N}+ \frac{1}{p}e) \geqslant \frac{\gamma(w_{k,N}- v_{k,N})+1}{p} \geqslant \frac{\gamma(w_{k,N})- \gamma(v_{k,N})+1}{p} 
\end{align}
As a simple consequence of the definition \eqref{eq:35} and the condition $\alpha, \beta \in Q$, we have,
\begin{align}
\frac{\theta}{2(1-\theta \gamma(w_{k,N})-(\phi-\theta)\gamma(v_{k,n}))} & = \frac{\bar{\alpha}}{2(1-\alpha-\beta)} = \frac{\alpha}{2(1-\alpha-\beta)} \frac{1}{\gamma(u_{k,N})} \leqslant \frac{1}{\gamma(u_{k,N})}  \nonumber \\
\frac{\phi}{2(1-\theta \gamma(w_{k,N})-(\phi-\theta)\gamma(v_{k,n}))} & = \frac{\bar{\beta}}{2(1-\alpha-\beta)} = \frac{\beta}{2(1-\alpha-\beta)} \frac{1}{\gamma(v_{k,N})} \leqslant \frac{1}{\gamma(v_{k,N})} \label{eq:46}
\end{align}
Now, as $\theta (w_k)_j+ (\phi-\theta)(v_k)_j \leqslant \theta \gamma(w_{k,N}) + (\phi-\theta)\gamma(v_{k,N}) $ for all $j \in N$, using Lemma \ref{lemma-5} and equation \eqref{eq:46}, we have,
\begin{align}
- \sum\limits_{j \in N} & \log \left(1- \theta (w_k)_j -  (\phi- \theta) (v_k)_j\right) \nonumber \\
&   \leqslant \ \theta \delta_k + (\phi-\theta) \omega_k + \frac{\|\theta w_{k,N}+ (\phi -\theta) v_{k,N}\|^2}{2(1-\theta \gamma(w_{k,N})-(\phi-\theta)\gamma(v_{k,n}))} \nonumber \\
& \leqslant \ \theta \delta_k + (\phi-\theta) \omega_k + \frac{\theta \|w_{k,N}\|^2 }{\gamma(u_{k,N})} + 2\frac{\epsilon(\phi-\theta) \|w_{k,N}\|}{\gamma(u_{k,N})} + \frac{(\phi-\theta)^2 \sigma_k^2 }{\theta \gamma(u_{k,N})} \label{eq:47}
\end{align}
Now, by combining equation \eqref{eq:41} and \eqref{eq:47}, we have,
\begin{multline}
F_N(x_{k+1})- F_N(x_k)   \ \leqslant \ \theta \|w_{k,N}\|^2 [-p + \frac{1}{\gamma(u_{k,N})}] + 2 \epsilon \|w_{k,N}\| (p \theta + \frac{\phi -\theta}{\gamma(u_{k,N})}) \\
+ \sigma_k^2 (-p \theta + \frac{(\phi -\theta)^2}{\theta \gamma(u_{k,N})})- p \theta \epsilon_k - p \theta \delta_k + (\phi + \theta) \omega_k - p \phi h_k \label{eq:48}
\end{multline}
Using the lower bound of equation \eqref{eq:45}, we can easily find,
\begin{align*}
-p \theta + \frac{\theta}{\gamma(u_{k,N})} \ & \leqslant \ - p \theta \frac{\gamma(w_{k,N}) - \gamma(v_{k,N})}{1+\gamma(w_{k,N}) - \gamma(v_{k,N})} = - p \bar{a} \\
p \theta + \frac{\phi -\theta}{\gamma(u_{k,N})} \ & \leqslant \ p \frac{\phi + \theta(\gamma(w_{k,N}) - \gamma(v_{k,N}))}{1+\gamma(w_{k,N}) - \gamma(v_{k,N})} = p \bar{b}
\end{align*}
Where, by the definition of  $\bar{a}$ and $\bar{b}$ given above, we can show that $\bar{a}, \bar{b} > 0$ are both finite constants. Now, from Theorem \ref{theorem-5}, we see that $\sum\limits_{k =L}^{\infty} (F_N(x_{k+1})- F_N(x_k)) > - \infty$. Also from Theorem \ref{theorem-9}, we get the following relation,
\begin{align}
\label{eq:49}
\sum\limits_{k =L}^{\infty} \left(|\delta_k| + \epsilon_k + \omega_k + h_k + \sigma_k^2\right) < \infty
\end{align}
Considering equations \eqref{eq:43}, \eqref{eq:48} and \eqref{eq:49}, for all $k \geqslant L$, we have,
\begin{align}
 & \ \text{Case 1:} \quad \quad  \sum\limits_{k =L}^{\infty} \|w_{k,N}\|^2 - 2 \epsilon \sum\limits_{k =L}^{\infty} \|w_{k,N}\| \ < \ \infty \label{eq:50} \\
 & \ \text{Case 2:} \quad \quad  \bar{a} \sum\limits_{k =L}^{\infty} \|w_{k,N}\|^2 - \bar{b} \sum\limits_{k =L}^{\infty} \|w_{k,N}\| \ < \ \infty \label{eq:51}
\end{align}
Both of the above cases, equation \eqref{eq:50} and \eqref{eq:51} imply that either the sequence $\{\gamma(w_{k,N})\}$ has a strictly positive/negative cluster point or $\lim_{k \rightarrow \infty }\gamma(w_{k,N}) = 0 $. If $\{\gamma(w_{k,N})\}$ has a cluster point then we must have, $\|w_{k,N}\| \rightarrow 0$.  Now, since $e^Tw_{k,N} = \delta_k$ and $\delta_k \rightarrow 0$, this implies whenever $\lim_{k \rightarrow \infty }\gamma(w_{k,N}) = 0 $, we must have, $w_{k,N} \rightarrow 0$. Either way, we have the following relationship,
\begin{align}
& \lim_{k \rightarrow \infty} w_{k,N} = 0  \ \ \Rightarrow \  \lim_{k \rightarrow \infty} u_{k,N} = \frac{1}{p} e \label{eq:52}
\end{align}
Now, for each $j \in N$, consider the sequences $\{\frac{(x_k)_j}{c^Tx_k-c^Tx^*}\}, \{\frac{(x_k)_i-(x^*)_i}{c^Tx_k-c^Tx^*}\}, \{y_k\}$ and $\{s_k\}$ for each $i \in B$. Let $s_{p_k} \rightarrow s^*$ for some sub-sequence $\{p_k\}$ of $k$. Since all of them are bounded for all $1 \leqslant i,j \leqslant n$. Thus using equation \eqref{eq:52}, we have,
\begin{align*}
 y_{p_k} \rightarrow y^*, \quad s_{p_k} \rightarrow s^*, & \quad \frac{p(x_{p_k})_j}{c^Tx_{p_k}-c^Tx^*} \rightarrow a_j \quad \text{for each} \quad j \in N \\
& p\frac{(x_{p_k})_i-(x^*)_i}{c^Tx_{p_k}-c^Tx^*} \rightarrow b_j \quad \text{for each} \quad j \in B 
\end{align*}
Considering equation \eqref{eq:13}, we know that $a_j > 0$ for all $j \in N$ and $(w_{p_k})_j \rightarrow 0$, for all $j \in N$ we have the following,
\begin{align*}
(s_{p_k})_j = \frac{c^Tx_{p_k}-c^Tx^*}{(x_{p_k})_j} \left((w_{p_k})_j - \frac{1}{p}\right) \rightarrow \ \frac{1}{a_j}
\end{align*}
Since from the definition of $D$, $A_Nx_{k,N}+ A_Bx_{k,B} = A_Bx_B^* + A_N * 0 = A_Bx_B^*$ holds and taking the limits, we see that $A_Na+ A_B b = 0$. This implies that $s_j = \frac{1}{a_J}$, for each $j \in N$ and $x = [x_B, x_N] = [-a, -b], \ s =[0, s_N^*], \ y = y^*$ solve the corresponding \textit{Karush Kahn Tucker} (KKT) conditions for the \textit{Analytic Center Problem} \eqref{eq:40}. Thus, $s_{p_k,B}$ converges to the analytic center for each sub-sequence, which in turn proves part (2) and (3) of Theorem \ref{theorem-10}. For proving the optimality, we notice that as $x^*, y^*$ and $s^*$ satisfies the primal and dual feasible criteria, respectively. They also satisfy   the complementary slackness property. Thus, $x^*, y^*$ and  $s^*$ are the optimal solutions for the respective primal and dual problems.
\end{proof}

\begin{remark}
Notice that, Theorem \ref{theorem-10} is a generalization of the original AFS algorithm. In this Theorem, if we consider $\beta = 0$ (without acceleration term), then it gives us the condition $\alpha \leqslant \frac{2}{3}$, which is in fact the respective bound for the original AFS (see \cite{saigal:1996}, \cite{tsuchiya:1995}).
\end{remark}
 
\begin{theorem}
\label{theorem-y} If $\alpha, \beta \in Q$, then there exist vectors $x^*, y^*,s^*$ such that the sequences $\{B(x_k)\}$, $\{y_k\}$ and  $\{s_k\}$ generated by the AAFS algorithm converges to $x^*, y^*$ and $s^*$ respectively, i.e.,
\begin{enumerate}
\item $B(x_k) \rightarrow x^*$,
\item $y_k \rightarrow y^*$,
\item $s_k \rightarrow s^*$
\end{enumerate}
Where  $x^*, y^*$ and $s^*$ are the optimal solutions of the respective primal and dual problems, and they also satisfy the strict complementary slackness property. Furthermore, the dual pair $(y^*,s^*)$ converges to the analytic center of the optimal dual face and the primal solution $x^*$ converges to the relative interior of the optimal primal face.
\end{theorem}
\begin{proof} From Theorem \ref{theorem-10}, we know that the sequence $\{x_k\}, \{y_k\}$ and  $\{s_k\}$ generated by the GAFS converges to $x^*, y^*$ and $s^*$, respectively. Then, using the definition \eqref{eq:4} and the basic idea of SST, we can  conclude that for all $j =1,2,...,n$ and $\alpha, \ \beta \in Q$ the following relation holds,
\begin{align*}
\lim_{k \to \infty} (B(x_k))_j = (x^*)_j, \quad \lim_{k \to \infty} y_k = y^*, \quad \lim_{k \to \infty} s_k = s^*
\end{align*}
Since, this holds for all $j =1,2,...,n$, we can prove $\lim_{k \to \infty} B(x_k) = x^*$. The last part is satisfied as we do not update the dual sequences at each iteration based on the sequence $\{B(x_k)\}$.
\end{proof}

\section{Numerical Experiments}
\label{sec:num}

In this section, we verified  the efficiency of the proposed variants of primal Affine Scaling algorithm presented in Section 2 through several numerical experiments. All of the experiments were carried out in a Intel Xeon Processor E5-2670, with double processors each with 20 MB cache, 2.60 GHz, 8.00 GT/s Intel QPI and 64 GB memory CPU. For simplicity of exposition, we considered three pairs of step sizes $(\alpha, \beta) = (0.4,0.2)$, $ (0.5,0.1)$ and $ (0.55,0.1)$, respectively for our experimental setup. We considered three types of LP problems: (1): \textit{Randomized Gaussian LP}, (2): \textit{Netlib Sparse LP} (real life instances \cite{netlib}) and (3): \textit{Randomized LP} with increasing $n$, constant $m$. We evaluated the performance of GAFS and AAFS with a long-step version of classical AFS. We considered duality gap tolerance $\epsilon $ as $10^{-3}$, $10^{-4}$ and $10^{-7}$ respectively and compared the results of our algorithms with the commercial LP solver (CPLEX-dual simplex \cite{cplex}) and with the \textit{MATLAB Optimization Toolbox} function \textit{fmincon} \footnote{For simplification, we used `fmin' in all of the figures and tables} \cite{fmincon}. The \textit{fmincon} function allows us to select `Interior Point' algorithm for a basic comparison as AFS is also an Interior Point Method.

\subsection{Comparison among AFS, GAFS and AAFS for dense data:}
\label{subsec:1}
The random dense data for these tests are generated as follows: All elements of the data matrix $A \in \R^{m \times n}$ and the cost vector $c \in \R^n$ are chosen to be \textit{i.i.d.} Gaussian $ \sim \mathcal{N} (-9,9)$. The right hand side $b \in \mathbb{R}^m$ is generated at random from corresponding distribution, but we made sure that $b \in \mathcal{R}(\mathbf{A})$. For that, we generated two vectors $x_1, x_2 \in \R^n$ at random from the corresponding distributions, then multiplied them by $A$ and set $b$ as a convex combination of those two vectors. We also made sure that the generated problems have bounded feasible solutions. We ran all algorithms 15 times and reported the averaged performance.
\vspace{- 0.2 cm}
\begin{figure}[H]
\begin{subfigure}{0.32\textwidth}
\includegraphics[width=\linewidth]{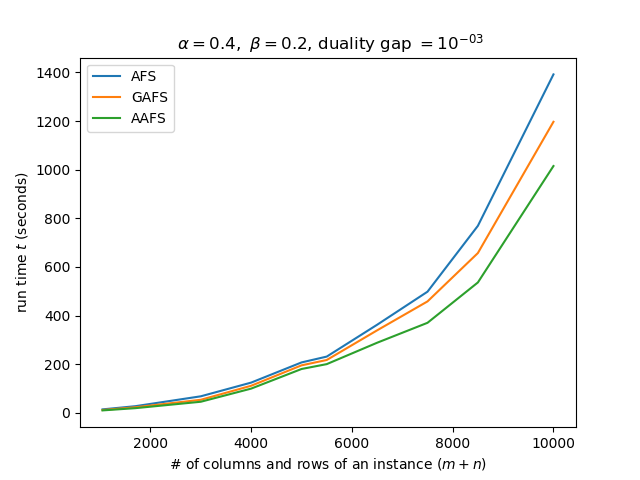}
\caption{$\alpha = 0.4, \ \beta = 0.2$}
\end{subfigure}
\hspace*{\fill}
\begin{subfigure}{0.32\textwidth}
\includegraphics[width=\linewidth]{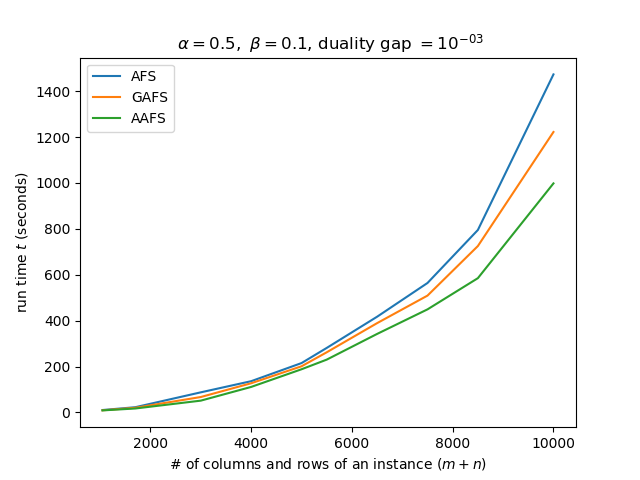}
\caption{$\alpha = 0.5, \ \beta = 0.1$}
\end{subfigure}
\hspace*{\fill}
\begin{subfigure}{0.32\textwidth}
\includegraphics[width=\linewidth]{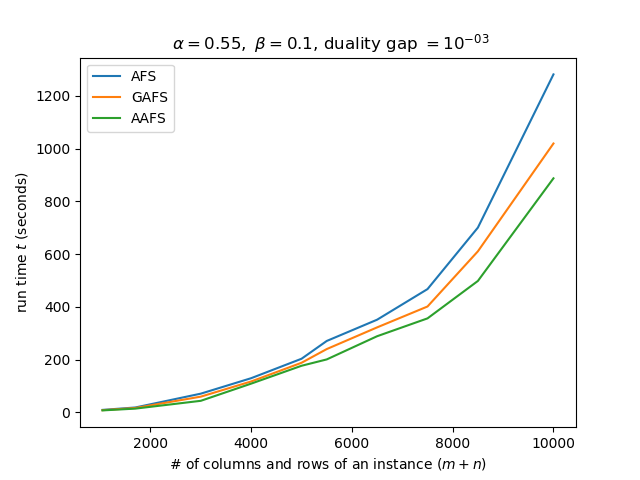}
\caption{$\alpha = 0.55, \ \beta = 0.1$}
\end{subfigure}
\caption{ Number of columns and rows ($m+n$) vs run time (duality gap $= 10^{-3}$, dense data)} \label{fig:1}
\end{figure}

\vspace{- 0.25 cm}

\begin{figure}[H]
\begin{subfigure}{0.32\textwidth}
\includegraphics[width=\linewidth]{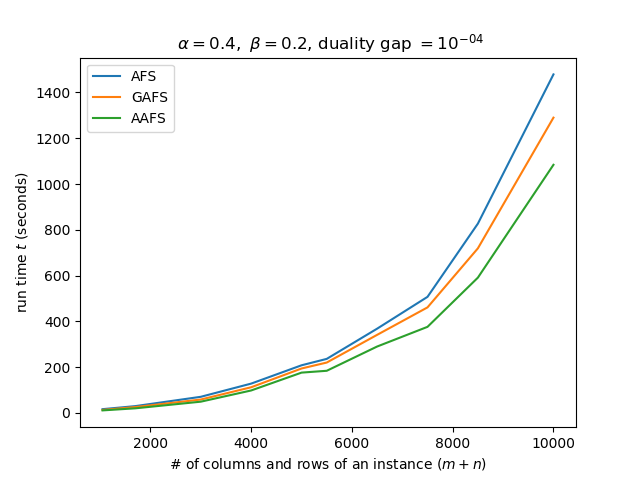}
\caption{$\alpha = 0.4, \ \beta = 0.2$}
\end{subfigure}
\hspace*{\fill}
\begin{subfigure}{0.32\textwidth}
\includegraphics[width=\linewidth]{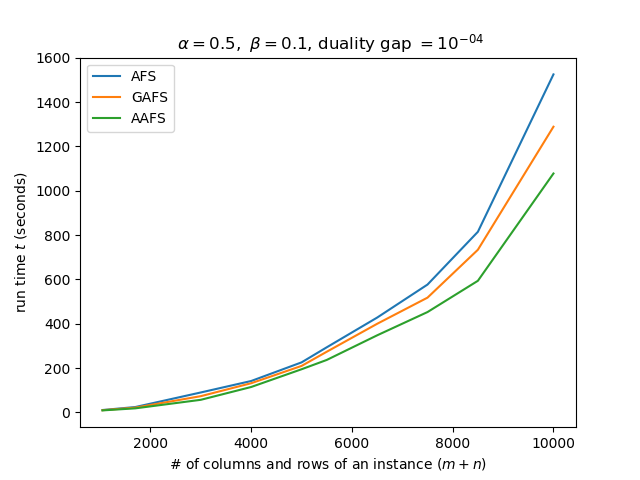}
\caption{$\alpha = 0.5, \ \beta = 0.1$}
\end{subfigure}
\hspace*{\fill}
\begin{subfigure}{0.32\textwidth}
\includegraphics[width=\linewidth]{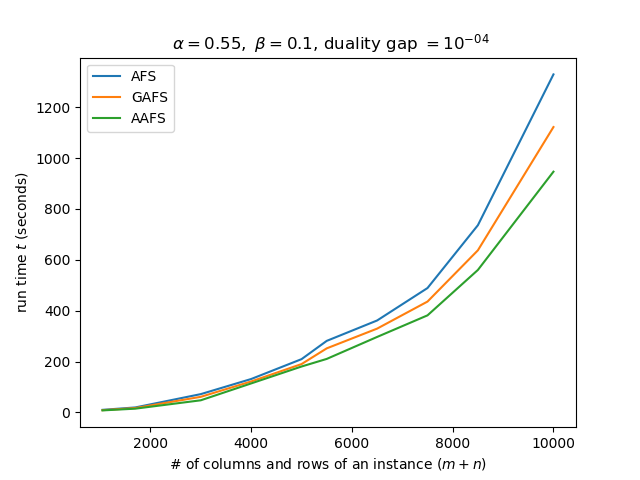}
\caption{$\alpha = 0.55, \ \beta = 0.1$}
\end{subfigure}
\caption{ Number of columns and rows ($m+n$) vs run time (duality gap $= 10^{-4}$, dense data)} \label{fig:2}
\end{figure}

\vspace{- 0.25 cm}

\begin{figure}[H]
\begin{subfigure}{0.32\textwidth}
\includegraphics[width=\linewidth]{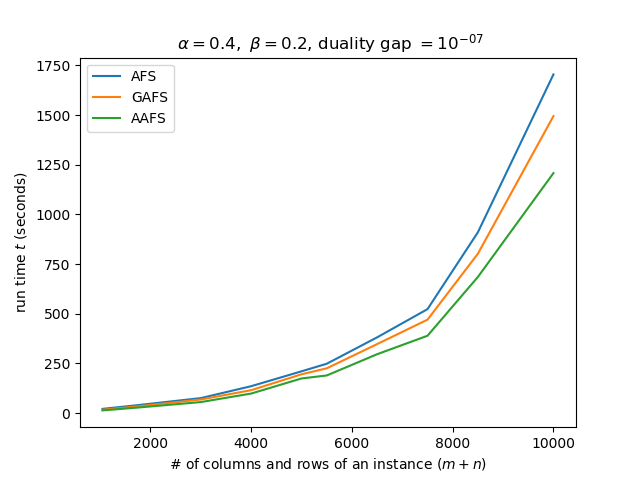}
\caption{$\alpha = 0.4, \ \beta = 0.2$}
\end{subfigure}
\hspace*{\fill}
\begin{subfigure}{0.32\textwidth}
\includegraphics[width=\linewidth]{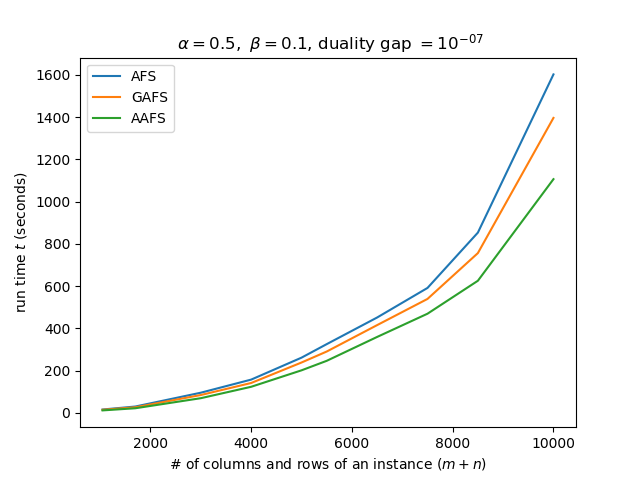}
\caption{$\alpha = 0.5, \ \beta = 0.1$}
\end{subfigure}
\hspace*{\fill}
\begin{subfigure}{0.32\textwidth}
\includegraphics[width=\linewidth]{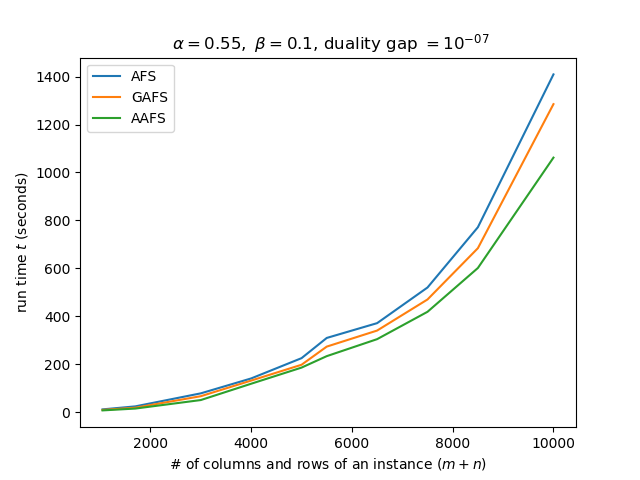}
\caption{$\alpha = 0.55, \ \beta = 0.1$}
\end{subfigure}
\caption{ Number of columns and rows ($m+n$) vs run time (duality gap $= 10^{-7}$, dense data)} \label{fig:3}
\end{figure}

In the above figures (Figure \ref{fig:1}, \ref{fig:2} and \ref{fig:3}), we compared AFS, GAFS and AAFS for different sets of $\alpha, \beta$ with diferent duality gaps. Our results show that the proposed variant algorithms reduced the runtime significantly. Furthermore, the reduction of runtime increases as the size of the instance gets larger (see Figure \ref{fig:1}, \ref{fig:2} and \ref{fig:3}). From figure \ref{fig:1}, \ref{fig:2} and \ref{fig:3}, we can conclude that the GAFS is faster than the original AFS irrespective of the size of the instances. Similarly, AAFS further accelerates the convergence of GAFS as the runtime decreases for all the instances. This is due to the integration of SST with the acceleration process and it converges much faster than the original AFS algorithm.

Now, we compared the performance of our proposed algorithms and AFS with the standard LP solvers \textit{fmincon} and `CPLEX-dual simplex'. For the comparison, we chose the best $\alpha, \beta$ pair from the above results ($\alpha = 0.55, \ \beta = 0.1$, validates our claim that `good result for larger $(\alpha+ \beta) \in Q$') and compared them for the duality gaps. At first, we presented the comparison graph for all the algorithms (AFS, GAFS, AAFS, \textit{fmincon} and CPLEX), where it is evident that performance of original AFS compare to \textit{fmincon} and CPLEX solver is very poor. However, our proposed acceleration scheme has significantly improve this performance gap. For  better understanding and fairness benchmark comparison, we compared our proposed AAFS with classical AFS and \textit{fmincon} (this comparison is fair as \textit{fmincon} uses the raw Barrier function method and CPLEX uses dual simplex \cite{cplex:h}, \cite{cplex}).
\vspace{-0.25 cm}
\begin{figure}[H]
\begin{subfigure}{0.32\textwidth}
\includegraphics[width=\linewidth]{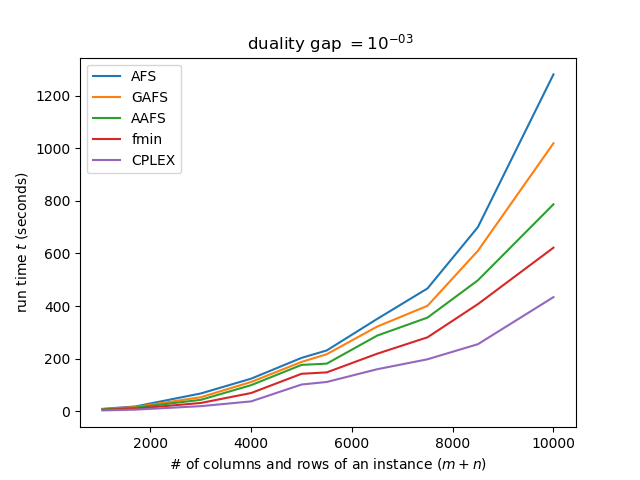}
\caption{duality gap $= 10^{-3}$}
\end{subfigure}
\hspace*{\fill}
\begin{subfigure}{0.32\textwidth}
\includegraphics[width=\linewidth]{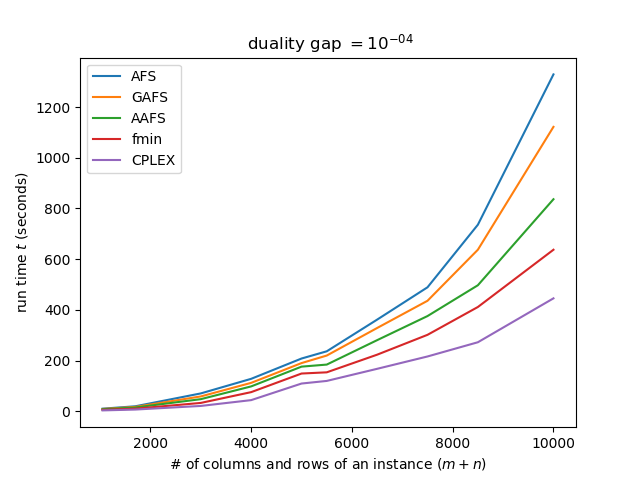}
\caption{duality gap $= 10^{-4}$}
\end{subfigure}
\hspace*{\fill}
\begin{subfigure}{0.32\textwidth}
\includegraphics[width=\linewidth]{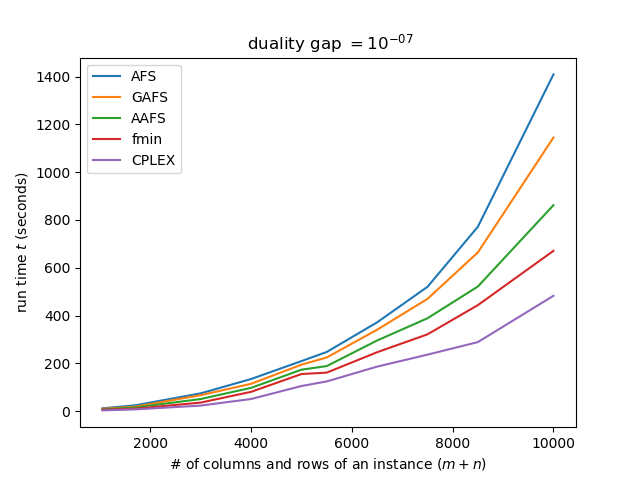}
\caption{duality gap $= 10^{-7}$}
\end{subfigure}
\caption{Comparison with `\textit{fmincon}' and `CPLEX' (dense data)} \label{fig:4}
\end{figure} 
\vspace{-0.25 cm}
\begin{figure}[H]
\begin{subfigure}{0.32\textwidth}
\includegraphics[width=\linewidth]{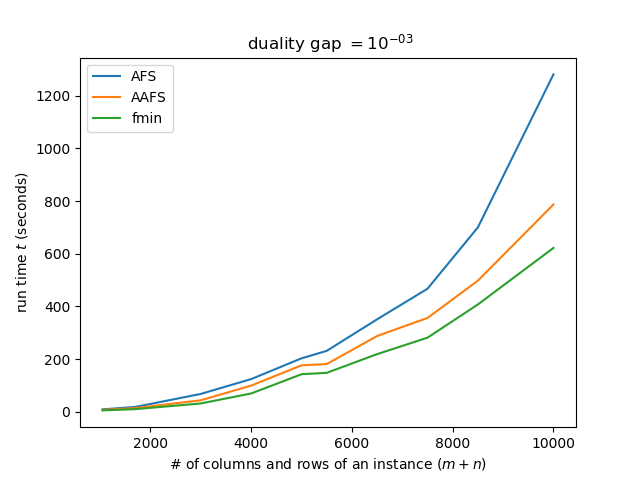}
\caption{duality gap $= 10^{-3}$}
\end{subfigure}
\hspace*{\fill}
\begin{subfigure}{0.32\textwidth}
\includegraphics[width=\linewidth]{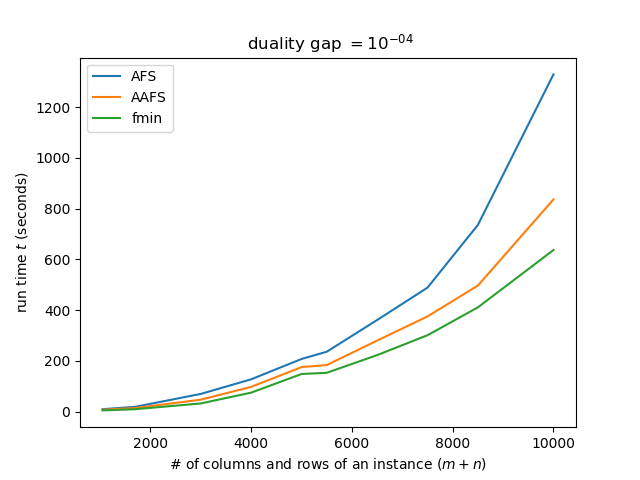}
\caption{duality gap $= 10^{-4}$}
\end{subfigure}
\hspace*{\fill}
\begin{subfigure}{0.32\textwidth}
\includegraphics[width=\linewidth]{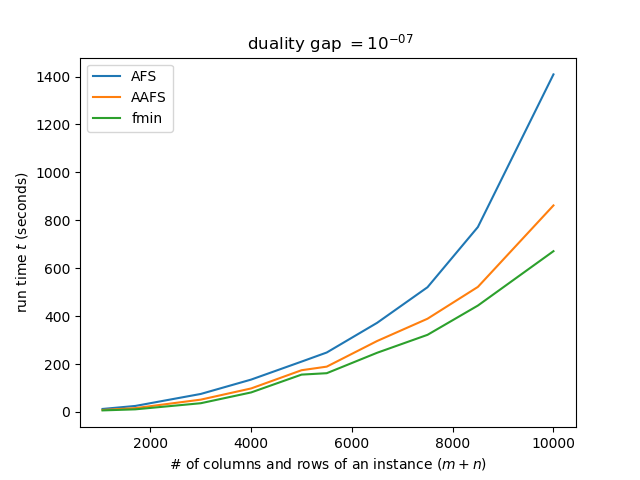}
\caption{duality gap $= 10^{-7}$}
\end{subfigure}
\caption{Comparison of `AAFS' with `AFS' and `\textit{fmincon}' (dense data)} \label{fig:5}
\end{figure}

\begin{table}[H]
\centering
\caption{Comparison among `GAFS', AAFS', `AFS and \textit{fmin}'$^a$}
\label{table:1 duality = -03}
\begin{tabular}{@{}ccccccc@{}}
\hline
Instances & \multicolumn{2}{c}{Dimension} & AFS & GAFS & AAFS & \textit{fmin} \\ \hline
 & $m$ & $n$ & Time & Time & Time & Time \\
1 & 400 & 650 & 9.02 & 7.96 & 7.11 & 5.07 \\
2 & 700 & 1000 & 18.16 & 16.22 & 13.87 & 9.87 \\
3 & 1000 & 2000 & 67.64 & 53.34 & 43.46 & 31.47 \\
4 & 1500 & 2500 & 123.95 & 111.40 & 99.30 & 69.34 \\
5 & 2000 & 3000 & 202.95 & 187.67 & 176.34 & 142.47 \\
6 & 2500 & 3000 & 231.20 & 216.90 & 181.03 & 147.79 \\
7 & 3000 & 3500 & 351.00 & 321.88 & 287.60 & 218.73 \\
8 & 3500 & 4000 & 467.07 & 401.00 & 356.10 & 281.20 \\
9 & 4000 & 4500 & 700.13 & 610.20 & 497.86 & 406.39 \\
10 & 4500 & 5500 & 1281.00 & 1019.00 & 787.20 & 619.31 \\ \hline
\end{tabular} \\
\footnotesize{$^a$ MATLAB Optimization Toolbox}
\end{table}

Based on the above figure (Figure \ref{fig:5}) and Table \ref{table:1 duality = -03}, we concluded that AFS takes on average \footnote{Note that, all comparison percentages reported in this work are calculated based on `\textit{fmincon}' CPU time} 75-80 $\%$ (\textit{min.} 40 $\%$, \textit{max.} 118 $\%$) more CPU time than \textit{fmincon} (Barrier method). However, our proposed AAFS takes on average 25-30 $\%$ (\textit{min.} 20 $\%$, \textit{max.} 45 $\%$) more CPU time than \textit{fmincon}. It is evident that the proposed acceleration reduced the CPU time consumption considerably (approximately on average 50 $\%$ reduction). The reason for this is that the classical AFS is an exponential time algorithm whereas the Barrier method is a polynomial time method. When we applied the proposed generalization and acceleration in AFS to generate GAFS and AAFS method respectively, they did well against Barrier function method as both methods uses the history information (i.e., AFS uses only $x_k$ to generate $x_{k+1}$, GAFS and AAFS uses $x_0,x_1,x_2,...,x_k$ to generate $x_{k+1}$).

\subsection{Comparison among AFS, GAFS and AAFS for sparse data}
\label{subsec:2}
In this subsection, we investigated the performance behaviour of classical AFS with the proposed GAFS and AAFS methods and also with MATLAB Optimization Toolbox function \textit{fmincon} \cite{fmincon} for several \textit{Netlib} LP instances (real life examples with sparse data \cite{netlib}). The experiment parameters remains the same as in the randomized instances. Figures \ref{fig:6}, \ref{fig:7} and \ref{fig:8} show the comparison graphs among AFS, GAFS and AAFS.

In the following figures, we compared AFS, GAFS and AAFS with the same parameters for the \textit{Netlib} LPs \footnote{For some instances, we slightly modify parameters $b$ and $c$ to make the instances solvable by our setup, for instance in some cases $b$ vector had some values as infinity which our setup can't handle, we replaced infinity with very large numbers}. Based on the figures below (Figure \ref{fig:6}, \ref{fig:7} and \ref{fig:8}), we can conclude that GAFS is faster than the original AFS and AAFS further accelerates the convergence of GAFS for all of the instances. Furthermore, one can notice that the runtime graphs do not follow the same trend as before (Figure \ref{fig:1}, \ref{fig:2} and \ref{fig:3}). The main reason for that is in these instances, we have another important parameter involved called sparsity (portion of nonzero entries in the matrix, $\delta = \frac{\text{nonzero entries of} \ A}{\text{total entries of} \ A}$) of the data matrix $A$. As shown in the following figures, sparsity affected the performance of the algorithms significantly.
\vspace{-0.2 cm}
\begin{figure}[H]
\begin{subfigure}{0.32\textwidth}
\includegraphics[width=\linewidth]{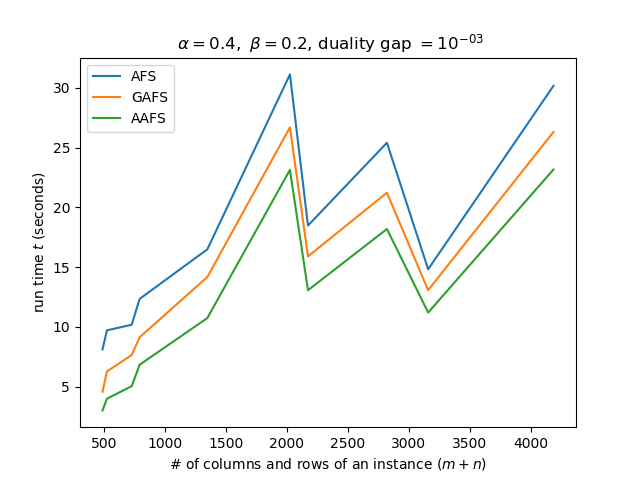}
\caption{$\alpha = 0.4, \ \beta = 0.2$}
\end{subfigure}
\hspace*{\fill}
\begin{subfigure}{0.32\textwidth}
\includegraphics[width=\linewidth]{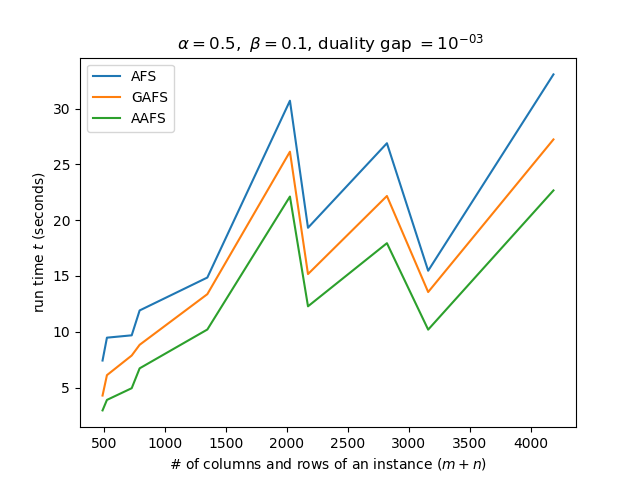}
\caption{$\alpha = 0.5, \ \beta = 0.1$}
\end{subfigure}
\hspace*{\fill}
\begin{subfigure}{0.32\textwidth}
\includegraphics[width=\linewidth]{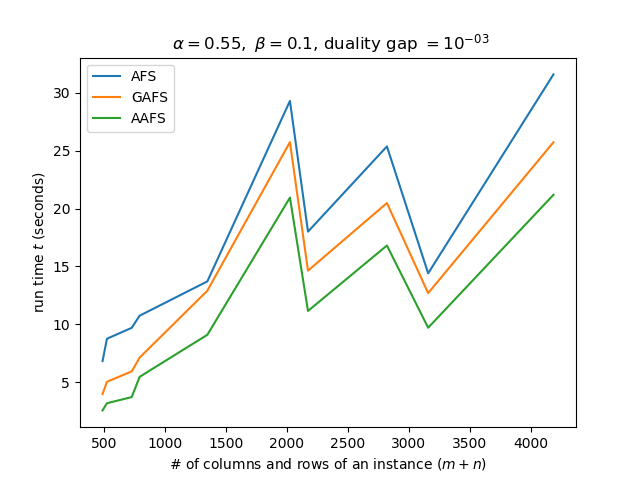}
\caption{$\alpha = 0.55, \ \beta = 0.1$}
\end{subfigure}
\caption{ Number of columns and rows ($m+n$) vs run time (duality gap $= 10^{-3}$, sparse data)} \label{fig:6}
\end{figure} 

\vspace{-0.35 cm}

\begin{figure}[H]
\begin{subfigure}{0.32\textwidth}
\includegraphics[width=\linewidth]{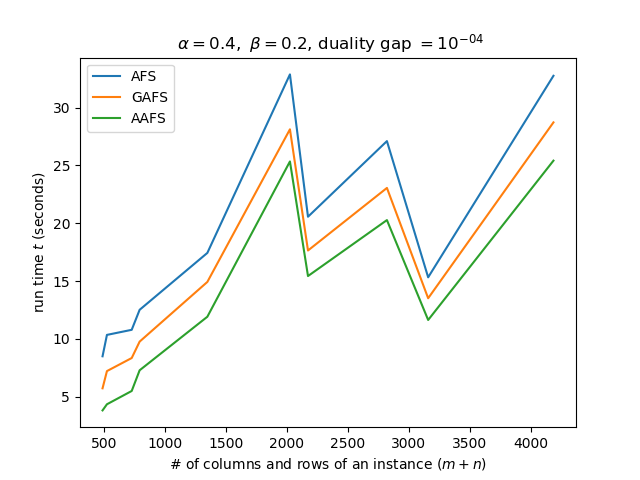}
\caption{$\alpha = 0.4, \ \beta = 0.2$}
\end{subfigure}
\hspace*{\fill}
\begin{subfigure}{0.32\textwidth}
\includegraphics[width=\linewidth]{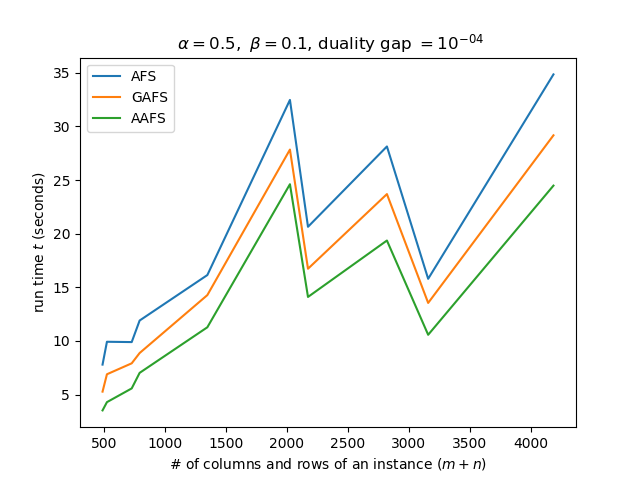}
\caption{$\alpha = 0.5, \ \beta = 0.1$}
\end{subfigure}
\hspace*{\fill}
\begin{subfigure}{0.32\textwidth}
\includegraphics[width=\linewidth]{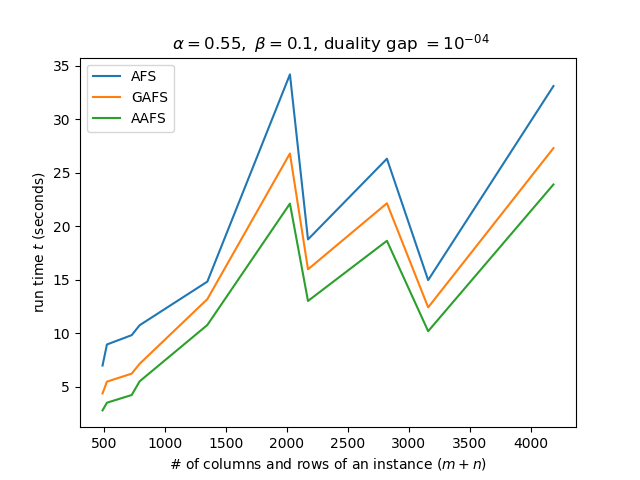}
\caption{$\alpha = 0.55, \ \beta = 0.1$}
\end{subfigure}
\caption{ Number of columns and rows ($m+n$) vs run time (duality gap $= 10^{-4}$, sparse data)} \label{fig:7}
\end{figure} 

\vspace{-0.35 cm}

\begin{figure}[H]
\begin{subfigure}{0.32\textwidth}
\includegraphics[width=\linewidth]{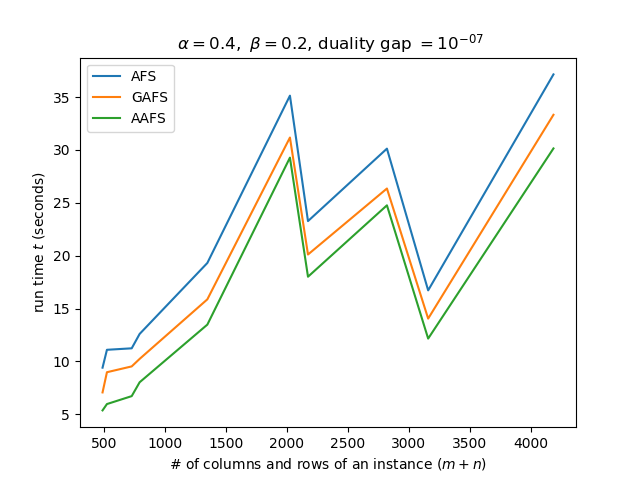}
\caption{$\alpha = 0.4, \ \beta = 0.2$}
\end{subfigure}
\hspace*{\fill}
\begin{subfigure}{0.32\textwidth}
\includegraphics[width=\linewidth]{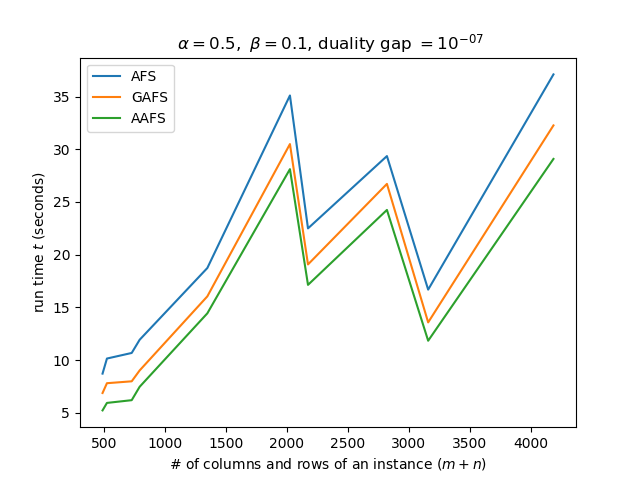}
\caption{$\alpha = 0.5, \ \beta = 0.1$}
\end{subfigure}
\hspace*{\fill}
\begin{subfigure}{0.32\textwidth}
\includegraphics[width=\linewidth]{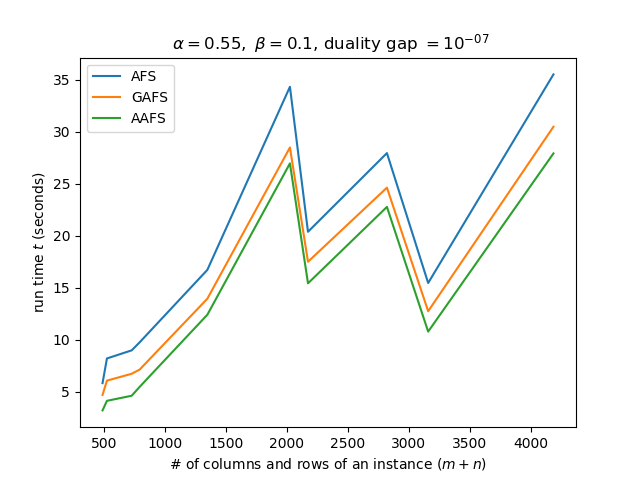}
\caption{$\alpha = 0.55, \ \beta = 0.1$}
\end{subfigure}
\caption{ Number of columns and rows ($m+n$) vs run time (duality gap $= 10^{-7}$, sparse data)} \label{fig:8}
\end{figure}

Now, we evaluated the performance of our proposed algorithms and AFS with the standard solver \textit{fmincon}. At first, we presented the comparison figure for all the algorithms (AFS, GAFS, AAFS and \textit{fmincon}). Then finally, for  better understanding, we compared our proposed AAFS with classical AFS and \textit{fmincon} (we explain in subsection \ref{subsec:1} that it is fair to compare with \textit{fmincon} as AFS, GAFS, AAFS and \textit{fmincon} uses the raw Interior Point Method whereas CPLEX uses dual simplex method \cite{cplex:h}, \cite{cplex}).
\vspace{-0.2 cm}
\begin{figure}[H]
\begin{subfigure}{0.32\textwidth}
\includegraphics[width=\linewidth]{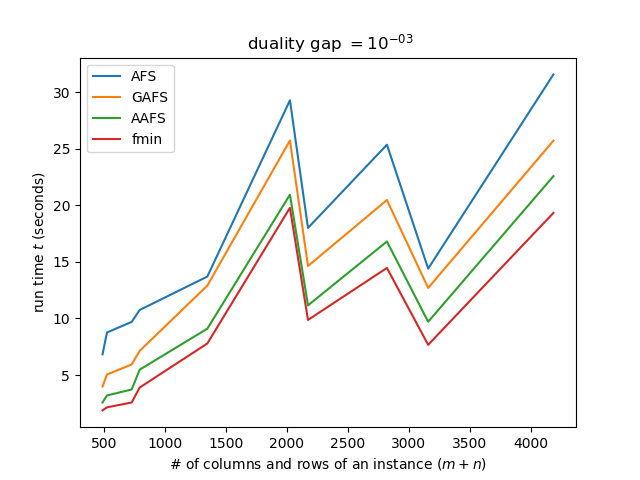}
\caption{$\text{duality gap} = 10^{-3}$}
\end{subfigure}
\hspace*{\fill}
\begin{subfigure}{0.32\textwidth}
\includegraphics[width=\linewidth]{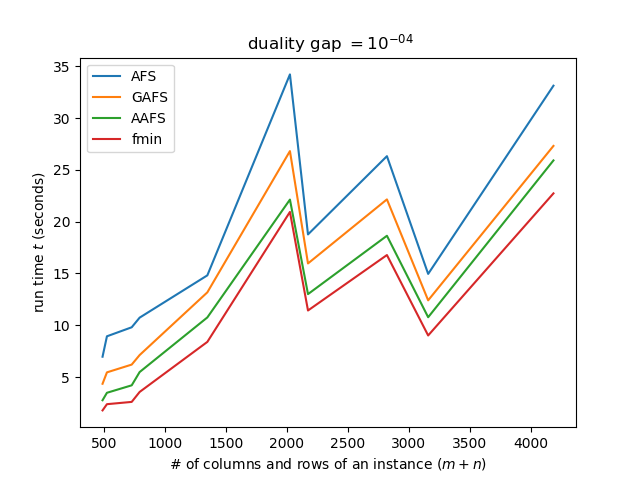}
\caption{$\text{duality gap} = 10^{-4}$}
\end{subfigure}
\hspace*{\fill}
\begin{subfigure}{0.32\textwidth}
\includegraphics[width=\linewidth]{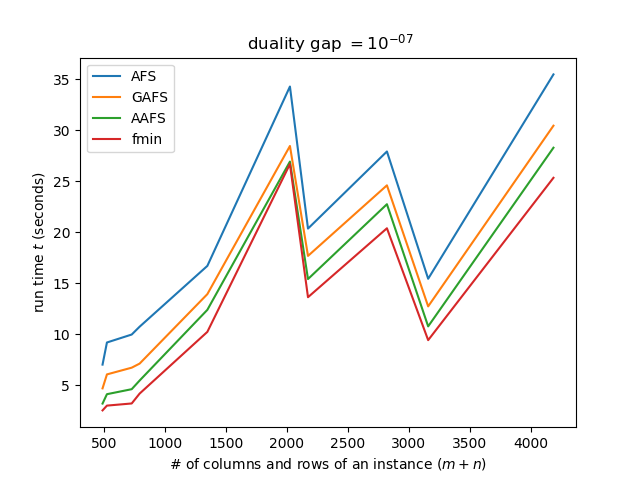}
\caption{$\text{duality gap} = 10^{-7}$}
\end{subfigure}
\caption{ Comparison of `GAFS' and `AAFS' with `AFS' and `\textit{fmincon}' (sparse data)} \label{fig:9}
\end{figure} 

\vspace{-0.2 cm}

\begin{figure}[H]
\begin{subfigure}{0.32\textwidth}
\includegraphics[width=\linewidth]{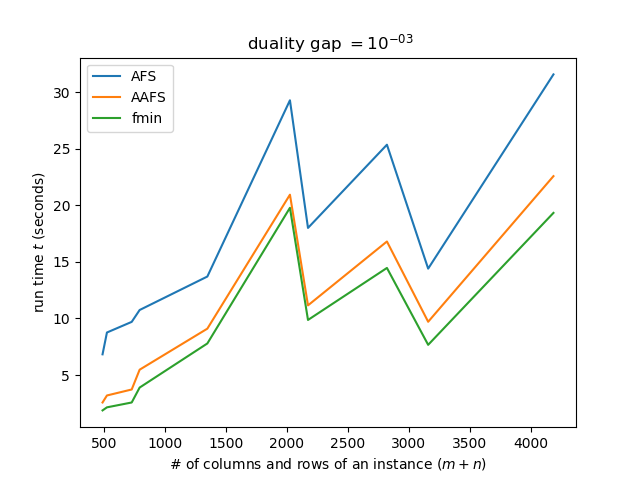}
\caption{$\text{duality gap} = 10^{-3}$}
\end{subfigure}
\hspace*{\fill}
\begin{subfigure}{0.32\textwidth}
\includegraphics[width=\linewidth]{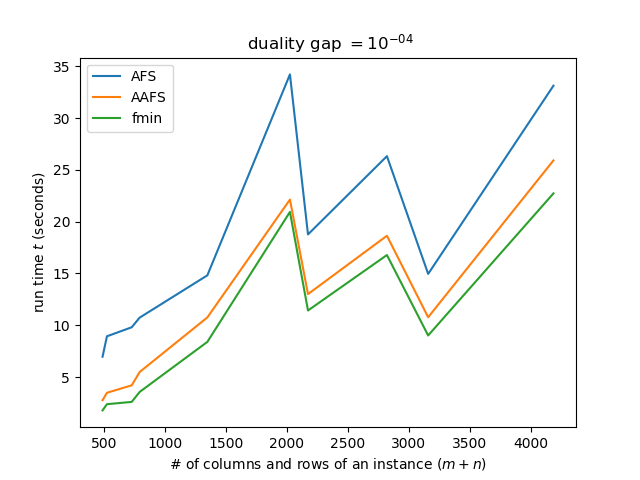}
\caption{$\text{duality gap} = 10^{-4}$}
\end{subfigure}
\hspace*{\fill}
\begin{subfigure}{0.32\textwidth}
\includegraphics[width=\linewidth]{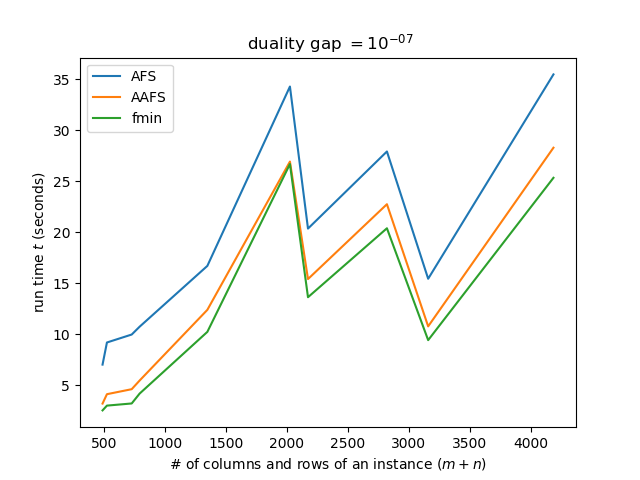}
\caption{$\text{duality gap} = 10^{-7}$}
\end{subfigure}
\caption{ Comparison of `AAFS' with `AFS' and `\textit{fmincon}' (sparse data)} \label{fig:10}
\end{figure} 

\begin{table}[H]
\small
\centering
\caption{Comparison of `GAFS' and `AAFS' with `AFS' and `\textit{fmin}$^a$}'
\label{table:2 random duality -03}
\begin{tabular}{@{}ccccccccc@{}}
\hline
Instances & \multicolumn{2}{c}{Dimension} & \begin{tabular}[c]{@{}c@{}}Nonzero\\ Entries\end{tabular} & Sparsity & AFS & GAFS & AAFS & \textit{fmin} \\ \hline
Title & $m$ & $n$ & $Z$ & $\delta$ & Time & Time & Time & Time \\
lp\_blend & 377 & 114 & 1302 & 0.0303 & 6.83 & 3.98 & 2.57 & 1.54 \\
lp\_adlittle & 389 & 138 & 1206 & 0.0225 & 8.76 & 5.04 & 3.19 & 1.97 \\
lp\_stocfor1 & 565 & 165 & 1359 & 0.0146 & 9.70 & 5.94 & 3.72 & 2.12 \\
lp\_recipe & 591 & 204 & 1871 & 0.0155 & 10.75 & 7.13 & 5.47 & 3.19 \\
lp\_brandy & 1047 & 303 & 5012 & 0.0158 & 13.71 & 12.91 & 9.1 & 7.15 \\
lp\_bandm & 1555 & 472 & 6097 & 0.0083 & 29.30 & 25.74 & 20.95 & 19.79 \\
lp\_scorpion & 1709 & 466 & 4282 & 0.0054 & 18.01 & 14.64 & 11.16 & 9.61 \\
lp\_agg & 2207 & 615 & 7085 & 0.0052 & 25.37 & 20.48 & 16.81 & 14.47 \\
lp\_degen2 & 2403 & 757 & 10387 & 0.0057 & 14.41 & 12.70 & 9.71 & 7.67 \\
lp\_finnis & 3123 & 1064 & 8052 & 0.0024 & 31.59 & 25.73 & 22.59 & 19.34 \\ \hline
\end{tabular}  \\
\footnotesize{$^a$ MATLAB Optimization Toolbox}
\end{table}
Based on the above figure (Figure \ref{fig:10}) and Table \ref{table:2 random duality -03}, we conclude that AFS takes on average 160-170 $\%$ (\textit{min.} 50 $\%$, \textit{max.} 357 $\%$) more CPU time than \textit{fmincon} (Barrier method). In comparison our proposed AAFS takes on average 30-35 $\%$ (\textit{min.} 16 $\%$, \textit{max.} 75 $\%$) more CPU time than \textit{fmincon}. It is evident that the proposed acceleration reduced the CPU time consumption considerably (approximately on average 130-135 $\%$ reduction). The reason for this is that we use the history information (i.e., $x_0, x_1, ..., x_k$ to update $x_{k+1}$) to develop GAFS and we use acceleration to GAFS to develop AAFS. By the construction of AAFS and GAFS, they should accelerate the convergence of AFS based on our convergence analysis, presented in Section \ref{sec:primal}. Another important consequence of acceleration is that the proposed AAFS is competitive with `\textit{fmincon}' for sparse instances (for some instances i.e., lp-bandm AAFS time consumption is much closer to `\textit{fmincon}' solver).

\subsection{Comparison between AFS, GAFS and AAFS for large $n$ and fixed $m$}
\label{subsec:3}

In this subsection, we considered the case where $m$ stays constant and $n$ grows exponentially. For better understanding, we consider $m = 50, 100$ and $n = 100,1000,10000,100000,1000000$ and run the experiments for duality gap $= 10^{-3}$ and $\alpha = 0.55, \beta = 0.1$. The comparison results of all algorithms (AFS, GAFS, AAFS and \textit{fmincon}) for this special instance is shown in Figure \ref{fig:11}.
\vspace{-0.25 cm}
\begin{figure}[H]
\begin{subfigure}{0.45\textwidth}
\includegraphics[width=\linewidth]{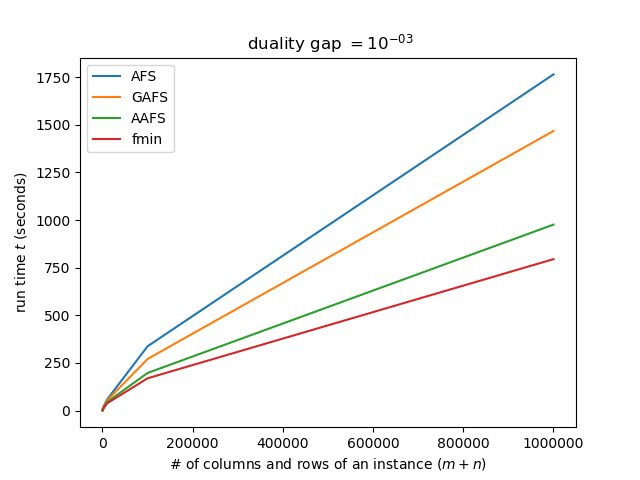}
\caption{$m = 50$}
\end{subfigure}
\hspace*{\fill}
\begin{subfigure}{0.45\textwidth}
\includegraphics[width=\linewidth]{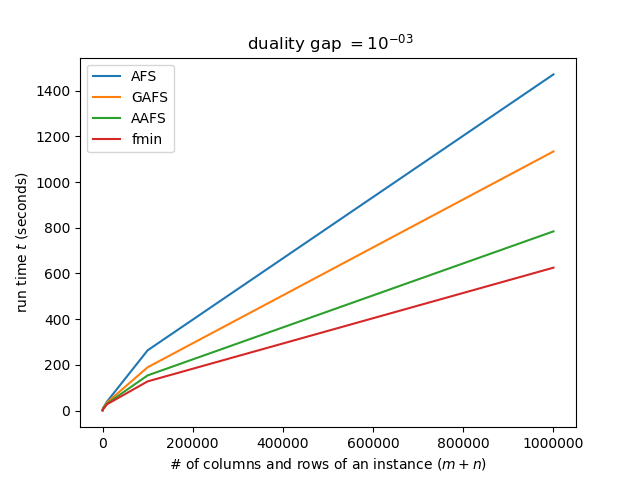}
\caption{$m = 100$}
\end{subfigure}
\caption{ Comparison of GAFS, AAFS, AFS and \textit{fmincon} while $m$ stays constant, $n$ increasing} \label{fig:11}
\end{figure} 

One interesting fact can be noted from Figure \ref{fig:11} is that the runtime graph follows the logarithmic trend as opposed to the exponential trend obtained in randomized instances. The main reason for this phenomenon is that though $n$ is increasing and total number of entries in the data-set $A$ is also increasing, the computational cost for computing the term $AX(AX^2A^T)^{-1}XA^T$ is cheap ($A \in \R^{m \times n}, \ AA^T \in \R^{m \times m}$, for small $m$, $AA^T$ inversion is cheap compared to the other cases). Furthermore, AAFS and GAFS both algorithms outperformed the classical AFS which supports the claim of this work. For a better understanding of the comparison, we plotted AFS, AAFS and \textit{fmincon} and the result is shown in Figure \ref{fig:12}. Based on Figure \ref{fig:12}, we can conclude that AFS takes on average 130-140 $\%$ (\textit{min.} 54 $\%$, \textit{max.} 240 $\%$) more CPU time than \textit{fmincon} (Barrier method). In comparison, our proposed AAFS takes on average 28-35 $\%$ (\textit{min.} 12 $\%$, \textit{max.} 52 $\%$) more CPU time than \textit{fmincon}. It is evident that the proposed acceleration reduced the CPU time consumption considerably (approximately on average 100-105 $\%$ reduction).
\vspace{-0.25 cm}
\begin{figure}[H]
\begin{subfigure}{0.45\textwidth}
\includegraphics[width=\linewidth]{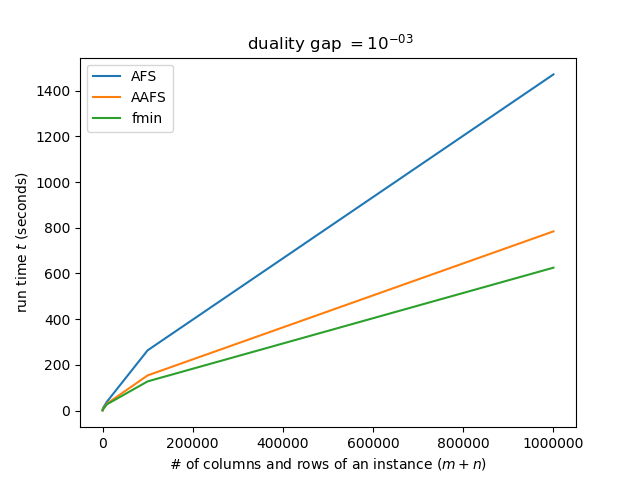}
\caption{$m = 50$}
\end{subfigure}
\hspace*{\fill}
\begin{subfigure}{0.45\textwidth}
\includegraphics[width=\linewidth]{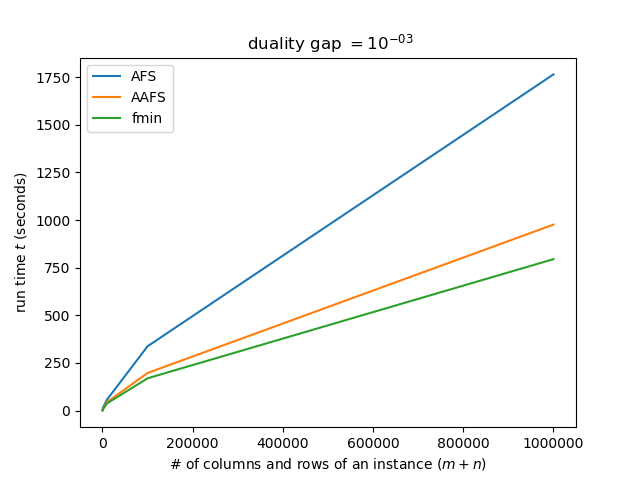}
\caption{$m = 100$}
\end{subfigure}
\caption{ Comparison of GAFS,  AFS and \textit{fmincon} while $m$ stays constant, $n$ increasing} \label{fig:12}
\end{figure} 

\noindent Our main goal of the numerical experiments is to show that the proposed GAFS and AAFS accelerate the classical AFS and support the claim proven in Section \ref{sec:primal}.  From the numerical results presented in Subsections \ref{subsec:1}, \ref{subsec:2} and \ref{subsec:3}, it is evident that the GAFS works faster than the classical AFS and AAFS outperforms AFS and GAFS for all of the instances. Apart from that, we also compared the proposed algorithms with standard LP solvers like \textit{fmincon} and CPLEX. Although the proposed GAFS and AAFS did not outperform the commercial LP solvers but in comparison with AFS, AAFS reduces the CPU time considerably (approximately on average 93-97 $\%$) for all of the instances. And almost for all of the instances, AAFS CPU time consumption is within (approximately 20-25 $\%$) of \textit{fmincon} solver. Proposed accelerated work show evidence of potential opportunity for applying acceleration to other Interior Point methods (i.e., Barrier method). Furthermore, it is a natural question to ask whether GAFS and AAFS require more computational cost compared to the original AFS for the additional acceleration effort. Since, the extra term $\frac{\beta}{\|X_k^{-1}(x_k-x_{k-1})\|_{\infty}}(x_k-x_{k-1})$ in GAFS requires only $O(n^2)$ algebraic operations and the extra term $B(x_k)$ in AAFS requires only $O(n^3)$ algebraic operations, both GAFS and AAFS require at most $O(n^3)$ algebraic operations at each iterations which make them computationally cheap. Benefits gained from the proposed accelerated techniques offset this additional computational effort. While the original AFS algorithm uses the current update to find the next update, the proposed algorithms use all the previous updates to find the next update (see Theorem \ref{theorem-1}) and thus the proposed generalized algorithm runs faster than the original algorithm.

\section{Conclusion}
\label{sec:conc}
In this work, we proposed two Affine Scaling algorithms for solving LP problems. The first algorithm (GAFS) integrated Nesterov's restarting strategy with the AFS method. Here, we introduced an additional residual term to the extrapolation step and determined the acceleration parameter $\beta$ adaptively. The proposed algorithm also generalizes the original AFS algorithm in the context of an extra parameter (i.e., the original AFS has $\beta =0$). The second algorithm (AAFS) integrated Shanks non-linear acceleration technique with the update of GAFS. Here, we introduced entry-wise SST to accelerate the process of GAFS. It is evident from our numerical experiments that the proposed AAFS and GAFS outperformed the classical AFS algorithm in comparison with standard LP solver \textit{fmincon} (AFS takes approximately more than 121-130 $\%$ of CPU time compared to \textit{fmincon} whereas in comparison AAFS takes approximately more than 20-25 $\%$). In terms of theoretical contribution, our proposed GAFS and AAFS revealed some interesting properties about the convergence characteristics of Affine Scaling method in general. Based on our analysis it is evident that the convergence criterion for AFS “$\alpha \leq \frac{2}{3}$” is a universal bound as GAFS and AAFS satisfy a more generalized bound “$\alpha + \beta \leq \frac{2}{3}$”. Finally, we believe that standard LP solvers can adapt acceleration in Barrier method (see below) for designing much more efficient solvers based on the theoretical and numerical results presented in this work.\\

\textbf{Future research:} Based on our theoretical and numerical analysis of GAFS and AAFS, it is evident that in future it is possible to discover a more general family of numerically efficient and theoretically interesting Affine Scaling methods and our work can help researchers to look for these types of methods for other IPMs. Moreover, based on our theoretical analysis, we believe these types of efficient Affine Scaling variants can also be designed for the following class of problems: Semi Definite Programming \cite{vanderbei:1999}, Nonlinear Smooth Programming \cite{wang:2009}, Linear Convex Programming \cite{cunha:2011}, Support Vector Machine \cite{maria:2011}, Linear Box Constrained Optimization \cite{wang:2014}, Nonlinear Box Constrained Optimization \cite{huang:2017}. This is due to the similarity of these methods in terms of algorithmic structure, i.e., the only difference between AFS and the above-mentioned AFS variants \cite{vanderbei:1999,wang:2009,cunha:2011,maria:2011,wang:2014,huang:2017} is the defining formulas for the sequences $\{y_k\}$ and  $\{s_k\}$ (see Algorithm \ref{alg:acc AFS} in Section \ref{sec:afs}). Furthermore, recent developments in optimization literature shows that Affine Scaling scheme is quite competitive with the state-of-the-art techniques for Linear Box Constrained Optimization, Nonlinear Box Constrained Optimization and Support Vector Machine problems \cite{wang:2014,huang:2017,maria:2011}. Our convergence analysis will enrich optimization literature and serve as a theoretical basis for applying the proposed acceleration schemes to the above-mentioned algorithms.

Finally, the numerical results and the convergence analysis suggest that acceleration can also be applied to other efficient Interior Point methods (Barrier method/Path following method and Karmarkar method). Since AFS and Barrier method follow the same scheme (only difference is that the objective function defined in the EAP problem \ref{eq:2} is different for Barrier method), the convergence analysis will also hold for Barrier method. Note that, for the Affine Scaling method, we defined the EAP problem (equation \ref{eq:2} is Section \ref{sec:afs}) as follows:
\begin{align}
\label{eq:n}
\textbf{min} \quad & w = c^Td \nonumber\\
\textbf{s.t} \ & Ad =0 \\
& \| X^{-1}d \| \leqslant \alpha  \nonumber
\end{align}
The main difference among Affine Scaling, Karmarkar method and Path following/Barrier method is that the later two methods use the following objective functions $G(x+d,s)$ and $B_\mu (x+d)$ in the EAP problems respectively,
\begin{align*}
    w = G(x+d,s) & \overset{\underset{\mathrm{def}}{}}{=} q \log s^T(x+d) - \sum\limits_{j=1}^{n} \log (d_j+x_j)- \sum\limits_{j=1}^{n} \log s_j \\
    w = B_{\mu}(x+d) & \overset{\underset{\mathrm{def}}{}}{=} c^T(x+d) - \mu \sum\limits_{j=1}^{n} \log (d_j+x_j)
\end{align*}
Here, $(x,s)$ is the primal dual pair and $q, \mu$ are parameters. Now, instead of using the original functions, we approximate the above penalty functions as follows:
\begin{align*}
    G(x+d,s) & \approx G(x,s) + \nabla^T G(x,s) d = G(x,s) + \frac{q}{s^Tx} s - X^{-1} =: w \ \\
    B_{\mu}(x+d) & \approx B_{\mu}(x) + \left(c^T- \mu e^T X^{-1}\right) d + \frac{1}{2} \mu d X^{-2}d =: w
\end{align*}
We can easily find the respective optimal solution $d^*$ for the EAP problems for the Karmarkar method and the Barrier method as follows:
\begin{align}
\label{eq:d}
    d^*_{\text{Karmarkar}} & = - \alpha \frac{Xu}{\|u\|} ; \quad u = \left(I - XA^T \left(AX^2A^T\right)^{-1} AX\right) \left(\frac{q}{s^Tx} Xs - e\right) \nonumber\\
    d^*_{Barrier}(\mu) &  = \left(I - X^2A^T \left(AX^2A^T\right)^{-1}A\right) \left(Xe - \frac{1}{\mu}X^2 c\right) 
\end{align}
Finally, considering the optimal direction vectors in \eqref{eq:d}, we can write the update formula for the primal sequences as follows:
\begin{align*}
    & \text{Karmarkar:} \quad   x_{k+1} = z_k + d^*_{\text{Karmarkar}} \\
    & \text{Barrier:} \quad   x_{k+1} = z_k + d^*_{\text{Barrier}}
\end{align*}
Here, the sequence $z_k$ is defined in \eqref{eq:3a}. Using the similar setup of this work, we can construct the update formula for the respective dual sequences.
We also believe that by designing special iterated sequences for $\alpha_k, \beta_k$ with appropriate choice of sequences $\{\pi_k\}$ and $\{\tau_k\}$, one can design efficient quadratically convergent algorithms, i.e.,
\begin{align*}
    & \alpha_{k+1} =: \alpha_k + \pi_k \\
    & \beta_{k+1} =: \beta_k + \tau_k
\end{align*}

\section{Acknowledgements}
We thank Mr. Tasnim Ibn Faiz, MIE, Northeastern University for his help and advice on programming in AMPL. We also thank Mr. Md Saiful Islam, MIE, Northeastern University for his help in the Numerical section. Finally, the authors are truly grateful to the anonymous referees and the editor for their  valuable comments and suggestions in the revision process which have greatly improved the quality of this paper.

\appendix

\section*{Appendix A}
\label{appendix-sec1}

In this section, we presented some Lemmas and their proofs which are required for the convergence analysis (most of the Lemmas are well known in the literature but they need proof in our context).

\begin{lemma}
\label{lemma-6}
For $\alpha, \beta \in Q$, the sequences $\beta_{k}$ and $\gamma_k$ defined in \eqref{eq:7} have the following properties:
\begin{enumerate}
\item There exists an $L \geqslant 1$ such that $\beta_k < 1$ for all $k \geqslant L$.
\item The sequence $\gamma_k \rightarrow 0$ as $k \rightarrow \infty$.
\end{enumerate}
\end{lemma}

\begin{proof}
Let, $\lim_{k \rightarrow \infty} x_k = x^*$ and $N = \{j \ | \ (x^*)_j = 0\}$. Then, we have $\lim_{k \rightarrow \infty} (x_k)_j = 0$, for all $j \in N$. As we know, $|\beta| < 1$, there exists a $M > 0$ and $L \geqslant 1$ such that for all $k \geqslant L$, we have,
\begin{align*}
(x_{k-1})_j \ \geqslant \ M \beta^{k-1}, \  (x_{k})_j \ \leqslant \ M \beta^{k}
\end{align*}
Thus for all $j \in N$ and for all $k \geqslant L$, we have,
\begin{align*}
\|X_k^{-1} \left(x_{k-1}-x_k\right)\|_{\infty} \ \geqslant \ |\frac{(x_{k-1})_j}{(x_k)_j} -1| \ \geqslant \ |\frac{ M \beta^{k-1}}{M \beta^{k}} -1| = \frac{1- \beta}{\beta}
\end{align*}
Since, $\beta \in Q$, thus $\beta < \frac{1}{\phi}$ gives us $\frac{\beta^2}{1-\beta} < 1$, therefore, we have,
\begin{align*}
\beta_k = \frac{\beta}{\|X_k^{-1} \left(x_{k-1}-x_k\right)\|_{\infty}} \ < \ \frac{\beta^2}{1-\beta} \ < \ 1
\end{align*}
This proves part (1) of Lemma \ref{lemma-6}. For the second part, let $\tilde{\beta} = \frac{\beta^2}{1-\beta} < 1$, then it is easy to see that for all $k \geqslant L, m > 0$, we have,
\begin{align*}
\gamma_{k+m} \ = \ \prod_{i =1}^{k-1} \beta_{i}. \prod_{j =k}^{k+m} \beta_{j} \ <   \ \prod_{i =1}^{k-1} \beta_{i}. \left(\tilde{\beta}\right)^{m} \rightarrow 0 \ \text{as} \ m \rightarrow \infty
\end{align*}
This proves  $\gamma_k \rightarrow 0$ as $k \rightarrow \infty$.
\end{proof}

\begin{lemma}
\label{lemma-7} 
There exists a $L_2 \geqslant 1$ such that for all $k \geqslant L_2$,
\begin{align*}
\frac{\gamma_k}{c^Tx_k-c^Tx^*} \ < \ L_2
\end{align*} 
\end{lemma}

\begin{proof}
From Lemma \ref{lemma-9}, we know that the sequence $\{u_k\}$ is bounded, it means there exists a $M_1 > 0$ such that for all $k$, we have $\|u_k\| \leqslant M_1$. Since, $\|X_ks_k\| > 0$, there exists a $\epsilon_2 > 0$ such that for all $k$, we have
$ \|X_ks_k\| > \epsilon_2 $. Similarly, as $\gamma_k \rightarrow 0 $ as $k \rightarrow \infty$, there exists a $L_3\geqslant 1$ such that for all $k \geqslant L_3$,
\begin{align*}
\gamma_k < \epsilon_2
\end{align*}
Combining these facts, for all $k \geqslant L_3$, we have,
\begin{align*}
r_k \ = \ \frac{\gamma_k}{c^Tx_k-c^Tx^*} = \frac{\gamma_k \|u_k\|}{\|X_ks_k\|} \ \leqslant \ \frac{\epsilon_2 M_1}{\epsilon_2} = M_1 
\end{align*}
Let, $ L_2 = \max\{M_1, r_1, r_2,..., r_{L_3}\} $, then for all $k$, we have,
\begin{align*}
r_k \ & \leqslant \ \max\{ r_1, r_2,..., r_{L_3}, r_{L_3+1}, ...\} \leqslant \ \max\{M_1, r_1, r_2,..., r_{L_3}\} = L_2
\end{align*}
Therefore, for all $k$ we have $r_k \ \leqslant \ L_2 $.
\end{proof}

\begin{lemma}
\label{lemma-8} 
If we define $G(k) = \sum\limits_{j = k}^{\infty} \gamma_j$, then there exists a $\bar{N} > 0$ such that for all $k \geqslant 0$ the following relation holds:
\begin{align}
\label{eq:18}
f_k \ = \ \frac{G(k)}{c^T x_k-c^Tx^*} \ \leqslant \ \bar{N}
\end{align}
\end{lemma}
\begin{proof} From Lemma \ref{lemma-6}, we know that there exists a $L \geqslant 1$ such that for all $k \geqslant L$, $\beta_k < 1$. Let $\tilde{\beta} = \max_{j \geqslant k}\{\beta_j\}$,  then from the definition of $G(k)$ for all $k \geqslant L$, we have,
\begin{align*}
f_k = \frac{G(k)}{c^T x_k-c^Tx^*} = & \frac{\gamma_k}{c^T x_k-c^Tx^*} (1+ \beta_{k+1}+ \beta_{k+1}\beta_{k+2}+ \beta_{k+1}\beta_{k+2}\beta_{k+3} + ...) \\
& \ < \ L_2 \sum\limits_{J=0}^{\infty} (\tilde{\beta})^j \ = \ \frac{L_2}{1-\tilde{\beta}}
\end{align*}
Let, $\bar{N} = \max \{f_1, f_2,..., f_{L}, \frac{L_2}{1-\tilde{\beta}}\}$, then for all $k \geqslant 0$, we have,
\begin{align*}
f_k = \frac{G(k)}{c^T x_k-c^Tx^*} \leqslant \max \{f_1, f_2,..., f_{L}, \frac{L_2}{1-\tilde{\beta}}\} \ = \ \bar{N}
\end{align*}
This proves Lemma \ref{lemma-8}.
\end{proof}

\begin{theorem}
\label{theorem-8} 
The sequences $\{v_k\}$ and $\{h_k\}$ satisfy the following properties:
Let $\epsilon > 0$, then there exists a $L \geqslant 1$ such that for all $k  \geqslant L$,
\begin{align*}
 & (a) \quad  \|v_{k,N}\|= \sigma_k, \ \ \sum\limits_{k = L}^{\infty} \sigma_k < \infty \quad \text{and} \quad  \sum\limits_{k = L}^{\infty} \sigma_k^2 < \infty \qquad \qquad \qquad \qquad \qquad \qquad \quad \quad \quad \quad \quad \quad \\
 & (b) \quad e^Tv_{k,N} = \omega_k, \ \ \sum\limits_{k = L}^{\infty} \omega_k < \infty \\
 & (c) \quad \sum\limits_{k = L}^{\infty} h_k < \infty \\
 & (d) \quad \text{There exists a} \ \ \epsilon_1 > 0 \ \text{such that}  \ \epsilon_1 \ \leqslant \ \gamma(v_{k,N}) \ \leqslant \ \epsilon \\
 & (e) \quad \|v_k\|_{\infty} = \|v_{k,N}\|_{\infty} = \gamma(v_{k}) = \gamma(v_{k,N})
\end{align*}
\end{theorem}

\begin{proof}
Part (a): We know from Theorem \ref{theorem-7}, there exist a $L_1 \geqslant 1$ such that for all $k \geqslant L_1$,
\begin{align}
 c^Tx_{k}- c^Tx^* \leqslant \left(1- \frac{\alpha}{\sqrt{n}}\right) (c^Tx_{k-1}- c^Tx^*) \leqslant \left(1- \frac{\alpha}{\sqrt{n}}\right)^{k-L_1} (c^Tx_{L_1}- c^Tx^*) \label{eq:A1}
\end{align}
As a consequence of \eqref{eq:A1}, for all $k \geqslant L_1$, we have,
\begin{align*}
\sum\limits_{k = L_1}^{\infty} \sigma_k = \sum\limits_{k = L_1}^{\infty} \|v_{k,N} \| & \ \leqslant \ \sum\limits_{k = L_1}^{\infty} \frac{\|X_k^{-1}(x_{k-1}-x_k)\|}{\|X_k^{-1}(x_k-x_{k-1})\|_{\infty}}(c^Tx_k-c^Tx^*) \\
& \ \leqslant \ \sqrt{n} \sum\limits_{k = L_1}^{\infty} (c^Tx_k-c^Tx^*)\\
& \ \leqslant \ \sqrt{n} (c^Tx_{L_1}- c^Tx^*) \sum\limits_{k = L_1}^{\infty} \left(1- \frac{\alpha}{\sqrt{n}}\right)^{k-L_1} \\
& \ = \ \frac{n(c^Tx_{L_1}- c^Tx^*) }{\alpha} \ < \ \infty
\end{align*}
Using the similar process, we can show that for all $k \geqslant L_1$,
\begin{align*}
 \sum\limits_{k = L_1}^{\infty} \sigma_k^2 \ =  \  \sum\limits_{k = L_1}^{\infty} \|v_{k,N} \|^2 \ < \ \infty
\end{align*}
This proves part (a) of Theorem \ref{theorem-8}. \\

\noindent Part (b): Using \eqref{eq:A1}, for all $k \geqslant L_1$, we have,
\begin{align*}
e^Tv_{k,N} = \frac{e^TX_{k,N}^{-1}(x_{k-1,N}-x_{k,N})}{\|X_k^{-1}(x_k-x_{k-1})\|_{\infty}}(c^Tx_k-c^Tx^*) \leqslant q \sqrt{n} ( c^Tx_{k}- c^Tx^* )
\end{align*}
Then we have,
\begin{align*}
\sum\limits_{k = L_1}^{\infty} \omega_k \ = \ \sum\limits_{k = L_1}^{\infty} e^Tv_{k,N} \ \leqslant \ q \sqrt{n} \sum\limits_{k = L_1}^{\infty} ( c^Tx_{k}- c^Tx^* ) < \ \frac{nq(c^Tx_{L_1}- c^Tx^*) }{\alpha} \ < \ \infty
\end{align*}
This proves part (b) of Theorem \ref{theorem-8}. \\

\noindent Part (c): From Lemma \ref{lemma-6}, we know there exist a $L_2 \geqslant 2$ such that for all $k \geqslant L_2$, 
\begin{align*}
\beta_k = \frac{\beta}{\|X_k^{-1} \left(x_{k-1}-x_k\right)\|_{\infty}} \ < \ \frac{\beta^2}{1-\beta} \ < \ 1
\end{align*}
As a consequence of the above relation, we have,
\begin{align*}
 \sum\limits_{k = L_2}^{\infty} h_k =  \sum\limits_{k = L_2}^{\infty}  \frac{c^T(x_{k-1}-x_k)}{\|X_k^{-1}(x_k-x_{k-1})\|_{\infty}} & = \sum\limits_{k = L_2}^{\infty} \frac{\beta_k}{\beta} (c^Tx_{k-1}-c^Tx_k) \\
 & < \  \frac{\beta}{1-\beta} (c^Tx_{L_2-1}-c^Tx^*) \ < \ \infty
\end{align*}
This proves part (c) of Theorem \ref{theorem-8}. \\

\noindent Part(d): Now, let $\epsilon > 0$. Since, $\lim_{k \rightarrow \infty} (c^Tx_k -c^Tx^*) = 0$ and $c^Tx_k -c^Tx^* > 0$, there exists an $\epsilon_1 > 0, L_3 \geqslant 1$ such that for all $k \geqslant L_3$, we have
$\epsilon_1 < \gamma(v_{k,N}) < \epsilon$, which is exactly what we want for part (d) of Theorem \ref{theorem-8}. \\

\noindent Part (e): Since, from the definition, we have $\lim_{k \rightarrow \infty} \frac{(x_{k-1})_j}{(x_{k})_j} = 1$, for all $j \in B$ and also $ \frac{(x_{k-1})_j}{(x_{k})_j} > 1$, for all $j \in N$, the required result follows as there exists a $L_4 \geqslant 1$ such that for all $k \geqslant L_4$, the following relation holds,
\begin{align*}
\|v_k\|_{\infty} \ = \ \|v_{k,N}\|_{\infty} \ = \ \gamma(v_{k}) \ = \ \gamma(v_{k,N})
\end{align*}
Finally, if we choose $L = \max\{L_1, L_2,L_3,L_4\}$ and combine all of the identities, we get part (e) of Theorem \ref{theorem-8}.
\end{proof}

\section*{Appendix B}
\label{appendix-sec2}

In this section, we provided some known results and their variants for the EAP problem given in \eqref{eq:2}. We discussed some properties of the sequences $\{x_k\}, \{y_k\}, \{s_k\}, \{d_k \} =\{X_k^2s_k\}$; generated by the GAFS discussed in Section \ref{sec:afs}. We provided some required proofs here as well. These properties also hold for the original AFS. We also describe some of the properties of the solution of the EAP problem \eqref{eq:2} as per our formulation. The EAP problem \eqref{eq:2} is similar to the one well studied in the literature (see \cite{saigal:1996}). Since the generalization parameter doesn't affect the EAP problem, we introduced some properties of EAP without providing proofs as most of the proofs are available in the literature (see \cite{dikin:1991}, \cite{saigal:1996}, \cite{tsuchiya:1995}).
As shown in several works by Saigal \cite{saigal:1996}, Vanderbei \textit{et al.} \cite{vanderbei:1986} and Dikin \cite{dikin:1967}, the solution $d^*$ of the EAP problem \eqref{eq:2} satisfies the following Lemma.

\begin{lemma}
\label{lemma-1}
Assume that the rows of $A$ are linearly independent and $c$ is not a linear combination of the rows of $A$. Let $x$ and  $z $ be some positive vectors with $Ax = Az =b$. Then the optimal solution $d^*$ of \eqref{eq:2} is given by
\[d^*= - \alpha \frac{X^2(c-A^Ty)}{\|X(c-A^Ty)\|} = - \alpha\frac{X^2s}{\|Xs\|} \quad \text{with} \quad y = \left(AX^2A^T\right)^{-1}AX^2c \]
Furthermore, the vector $\bar{x} = x+ d^* + \bar{\beta} (x-z)$ satisfies $A \bar{x} =b$ and $c^T \bar{x} < c^T x$.
\end{lemma}

\begin{proof}
We see that $d^*$ is the optimal solution of the EAP problem (2) (see \cite{vanderbei:1986}). Now, since $x$ and $z $ satisfy the condition, $Ax= Az = b$, then we have $A\bar{x} = Ax + Ad^* + \bar{\beta} (Ax-Az) = b$, which shows  $A\bar{x} =b$. Now for the last part, we have,
\begin{align*}
c^T \bar{x} = c^Tx - \alpha \frac{c^TX^2s}{\|Xs\|}+ \bar{\beta} (c^Tx-c^Tz) < c^Tx - \alpha \|Xs\| \ < \ c^Tx
\end{align*}
This proves the above Lemma. Here,  we use the identity $c^TX^2s = \|Xs\|^2$ (see Lemma \ref{lemma-2}).
\end{proof}

\begin{lemma}
\label{lemma-2} [Theorem 1 in \cite{saigal:1996}]
For all $k \geqslant 0$ the following identity holds:
\begin{align*}
c^Td_k = c^TX_k^2s_k = \|X_ks_k\|^2 = \|X_k^{-1}d_k\|^2
\end{align*}
\end{lemma}
\begin{proof}
At first, let us denote $P_k = I - X_k A^T\left(AX_k^2A^T\right)^{-1}AX_k$, it can be easily shown that $P_k$ is a projection matrix, (i.e., $P_k = P_k^T = P_k^2$). Thus, we have,
\begin{align*}
c^Td_k = c^TX_k^2s_k = & c^TX_k \left(I - X_k A^T\left(AX_k^2A^T\right)^{-1}AX_k\right)X_kc \\
& = \ c^TX_kP_kX_kc \ = \|P_kX_kc\|^2 \\
 \text{And}  \quad  P_kX_kc \ = & \ X_k \left(c-A^T\left(AX_k^2A^T\right)^{-1}AX_k^2c\right) = X_ks_k
\end{align*}
The proof is complete.

\end{proof}

\begin{lemma}
\label{lemma-3} [Theorem 4 in \cite{saigal:1996}]
For all $x > 0$, there exists a $q(A) > 0$ such that,
\begin{align*}
\|\left(AX^2A^T\right)^{-1}AX^2p\| \ \leqslant \ q(A) \|p\|
\end{align*}
\end{lemma}

\begin{lemma}
\label{lemma-4} [Corollary 6 in \cite{saigal:1996}]
For all $x > 0$, there exists a $p(A,c) > 0$ such that, if $\bar{d}$ solves EAP (2) then the following relationship holds:
\begin{align*}
\|\bar{d}\| \ \leqslant \ p(A,c) \ c^T \bar{d} = M c^T \bar{d}
\end{align*}
\end{lemma}

\begin{lemma}
\label{lemma-5} [Lemma 8 in \cite{saigal:1996}] Let $w \in \R^q$ and $0 <\lambda < 1$ be such that $w_j \leqslant \lambda$ then
\begin{align*}
\sum\limits_{i =1}^{q} \log(1-w_i) \geqslant -e^Tw- \frac{\|w\|^2}{2(1-\lambda)}
\end{align*}
\end{lemma}

\begin{lemma}
\label{lemma-9} 
It is a very well known Lemma in the literature for AFS methods (see Theorem 13 in \cite{saigal:1996}, Lemma 3.8 and 3.11 in \cite{tsuchiya:1995}). The sequence $\{u_k\}$ has the following properties:
\begin{enumerate}
\item It is bounded
\item There exists a $L \geqslant 1$ such that for all $k > L$,
\begin{align*}
 & (a) \ \|u_k\|^2 = \|u_{k,N}\|^2+\epsilon_k, \ \ \sum\limits_{k = L}^{\infty} |\epsilon_k| < \infty \quad \quad \qquad \qquad \qquad \qquad \qquad \qquad \quad \quad \quad \quad \quad \quad \\
 & (b) \ e^Tu_{k,N} = 1+ \delta_k, \ \ \sum\limits_{k = L}^{\infty} |\delta_k| < \infty \\
 & (c) \ \frac{1}{\alpha} \ \geqslant \ \gamma(u_{k,N}) \ \geqslant \ \frac{1}{2p} \vspace{- 0.15 cm}\\
 & (d) \ \gamma(u_{k})  \ = \ \gamma(u_{k,N}) 
\end{align*}
\end{enumerate}
\end{lemma}
\begin{proof}
The proof is similar to the one for AFS method given by Saigal \cite{saigal:1996}, as the direction $\frac{X_k^2s_k}{\|X_ks_k\|}$ generated by GAFS algorithm also satisfies the EAP problem defined in equation \eqref{eq:2} (for a detailed proof see Tsuchiya \textit{et al.} \cite{tsuchiya:1995}).
\end{proof}

\begin{lemma}{(Lemma 15 in \cite{saigal:1996})}
\label{lemma-11}
If the analytic center defined by the solution of problem \eqref{eq:40} exists, it is unique.
\end{lemma}

\bibliographystyle{ieeetr}
\bibliography{aafs}

\end{document}